\documentclass[10pt, twoside]{amsart}
\usepackage{amsmath, amssymb, amsthm}

\newtheorem{theo}{Theorem}[section]
\newtheorem{cor}[theo]{Corollary}
\newtheorem{dfn}[theo]{Definition}
\newtheorem{lemma}[theo]{Lemma}
\newtheorem{prop}[theo]{Proposition}
\newtheorem*{ques*}{Questions}
\theoremstyle{definition}
\newtheorem{exa}[theo]{Example}

\def\C{{\mathbb C}}
\def\Cal{{\rm Cal}}
\def\Diff{{\rm Diff}}
\def\Ham{{\rm Ham}}
\def\Hameo{{\rm Hameo}}
\def\Hel{{\mathcal H}}
\def\Homeo{{\rm Homeo}}
\def\id{{\rm id}}
\def\R{{\mathbb R}}
\def\Rel{{\mathcal R}}
\def\Symp{{\rm Symp}}
\def\Sympeo{{\rm Sympeo}}
\def\vol{{\rm vol}}
\def\Z{{\mathbb Z}}

\title[Helicity of strictly contact vector fields]{Helicity of vector fields preserving a regular contact form and topologically conjugate smooth dynamical systems}
\author[S.~M\"uller \& P.~Spaeth]{Stefan M\"uller and Peter Spaeth}
\email{mueller@kias.re.kr \text{\it{and}} spaeth@kias.re.kr}
\address{Korea Institute for Advanced Study, Seoul 130-722, Republic of Korea}


\begin{document}
\thispagestyle{plain}

\begin{abstract}
We compute the helicity of a vector field preserving a regular contact form on a closed three-dimensional manifold, and improve results by J.-M.~Gambaudo and \'E.~Ghys~\cite{gambaudo:ea97} relating the helicity of the suspension of a surface isotopy to the Calabi invariant of the latter.
Based on these results, we provide positive answers to two questions posed by V.~I.~Arnold~\cite{arnold:ahi86}.
In the presence of a regular contact form that is also preserved, the helicity extends to an invariant of an isotopy of volume preserving homeomorphisms, and is invariant under conjugation by volume preserving homeomorphisms.
A similar statement also holds for suspensions of surface isotopies and surface diffeomorphisms.
This requires the techniques of topological Hamiltonian and contact dynamics developed in \cite{mueller:ghh07, mueller:ghl08, viterbo:ugh06, buhovsky:ugh11, banyaga:ugh11, mueller:gcd11}.
	
Moreover, we generalize an example of H.~Furstenberg~\cite{furstenberg:set61} of topologically conjugate but not $C^1$-conjugate area preserving diffeomorphisms of the two-torus to trivial $T^2$-bundles, and construct examples of Hamiltonian and contact vector fields that are topologically conjugate but not $C^1$-conjugate.
Higher-dimensional helicities are considered briefly at the end of the paper.
\end{abstract}

\maketitle
\section{Introduction} \label{sec:intro}
According to Arnold~\cite{arnold:ahi86}, ``The asymptotic Hopf invariant is an invariant of a divergence-free vector field on a three-dimensional manifold with given volume element.
It is invariant under the group of volume-preserving diffeomorphisms, and describes the `helicity' of the field, i.e.\ the mean asymptotic rotation of the phase curves around each other.''
If $X$ is a divergence-free vector field on a closed (i.e.\ compact and without boundary) smooth three-manifold $M$, equipped with a volume form $\mu$, then the two-form $\iota_X \mu$ is closed.
Assuming it is exact, one may choose a primitive one-form $\beta_X$, and define the \emph{helicity} (or \emph{asymptotic Hopf invariant} or \emph{Arnold invariant}) of $X$ as the real number
	\[ \Hel (X) = \int_M \beta_X \wedge d\beta_X. \]
This number does not depend on the choice of $\beta_X$ with $d\beta_X = \iota_X \mu$.
Arnold in fact gives two more equivalent definitions of the helicity, one as an average asymptotic linking number of the flow lines of $X$, and the other one equal to $\int_M g (X,Y)$, where $g$ is some auxiliary Riemannian metric on $M$, and $Y$ is a divergence-free vector field satisfying the relation $\text{curl}_g Y = X$.
It is the first construction we shall use exclusively in this work.

The asymptotic Hopf invariant generalizes the classical Hopf invariant of (the homotopy class of) a map $S^3 \to S^2$.
Arnold and B.~A.~Khesin \cite{arnold:tmh98} note that, ``Although the idea of helicity goes back to Helmholtz and Kelvin (see [Kel]), its second birth in magnetohydrodynamics is due to Woltjer [Wol] and in ideal hydrodynamics is due to Moffatt [Mof1], who revealed its topological character (see also [Mor2]).
The word `helicity' was coined in [Mof1] and has been widely used in fluid mechanics and magnetohydrodynamics since then.''
In addition to the references cited above, we also recommend \cite{ghys:kd07} for further reading, and for more details on the definition and some of its applications.
The above publications also establish the basic properties of the helicity invariant, and contain additional interesting references.
See also Section~\ref{sec:helicity}.

It follows almost immediately from the definition that the helicity depends continuously on the vector field in the $C^1$-topology, and is invariant under conjugation by volume preserving $C^1$-diffeomorphisms.
When viewed as an invariant of the volume preserving isotopy $\{ \phi_X^t \}_{0 \le t \le 1}$ generated by the vector field $X$, the helicity is not continuous with respect to the $C^0$-topology.
We recall two questions posed by Arnold regarding the topological character of the helicity.

\begin{ques*}{\cite{arnold:ahi86}}\label{que:arnold}
(i) Is the helicity invariant under conjugation by a volume preserving homeomorphism?
More precisely, if $X$ and $Y$ are (exact) divergence-free vector fields, $\phi$ a homeomorphism that preserves the measure induced by $\mu$, and $\{ \phi_Y^t \} = \{ \phi \circ \phi_X^t \circ \phi^{-1}\}$, does the identity $\Hel (X) = \Hel (Y)$ hold?

(ii) If $\{ \phi_t \}_{0 \le t \le 1}$ is an isotopy of volume preserving homeomorphisms, can one define a number $\Hel (\{ \phi_t \})$ that extends the definition for smooth isotopies?
\end{ques*}

The main purpose of this article is to address these questions.
As a first step in that direction, in Section~\ref{sec:helicity-contact} we demonstrate the following.

\begin{theo} \label{thm:helicity}
Suppose a closed three-manifold $M$ admits a regular contact form $\alpha$, and equip $M$ with the canonical volume form $\alpha \wedge d\alpha$ induced by $\alpha$.
Let $X_H$ be a strictly contact vector field with contact Hamiltonian function $H = \alpha (X_H)$.
Then $X_H$ is exact divergence-free, and
	\[ \Hel (X_H) = \left( 4 c^2 (H) - 3 c (H^2) \right) \cdot \vol (M,\alpha \wedge d\alpha), \]
where the integer $\vol (M,\alpha \wedge d\alpha)$ denotes the total volume, and $c$ the average value of a function on $M$, both with respect to the canonical volume form.
\end{theo}

In a similar vein, Gambaudo and Ghys showed that the helicity of the suspension of a surface isotopy $\{ \phi_t \}$ is proportional to the Calabi invariant of $\{ \phi_t \}$ \cite{gambaudo:ea97}.
Using different techniques, we improve their result as follows.
See Section~\ref{sec:suspensions} for the relevant definitions.

\begin{theo}\label{thm:suspension}
Let $D^2$ be the unit disk in $\R^2$ with its standard area form $\omega$, and let $\{\phi_t\}_{0\le t\le 1}$ be a smooth area preserving isotopy of $D^2$ that is the identity near the boundary $\partial D^2$.
The helicity of the suspension $\tau_* X (\{ \phi_t \})$ with respect to the standard volume form $dV$ on $\R^3$ equals twice the Calabi invariant with respect to $\omega$ of the time-one map $\phi$ of the isotopy $\{ \phi_t \}$.
In fact,
	\[ \Hel \left( X + \frac{\partial}{\partial t} \right) = \Hel (X) + 2 \Rel \left( X, \frac{\partial}{\partial t} \right) + \Hel \left( \frac{\partial}{\partial t} \right),\]
and the first and last term vanish, while $\Rel (X,\frac{\partial}{\partial t}) = \Cal (\phi)$.
\end{theo}

Here $\Rel(\cdot, \cdot)$ denotes the relative helicity defined below in Section~\ref{sec:helicity}.
In particular, the helicity is not $C^0$-continuous with respect to the isotopy $\{ \phi_t \}$, because the Calabi invariant is not $C^0$-continuous with respect to $\phi$ (or $\{ \phi_t \}$) \cite{gambaudo:ea97}.
The latter by definition is the real number
\begin{align} \label{eqn:calabi}
	\Cal (\phi) = \int_0^1 \int_{D^2} F_t \, \omega \, dt,
\end{align}
where $F:[0,1]\times D^2 \to \R$ is the unique normalized smooth Hamiltonian function generating the Hamiltonian isotopy $\{\phi_t\} = \{\phi^t_F\}$, and only depends on the time-one map $\phi$.

See also Section~\ref{sec:suspensions} for the definition of the double suspension of a pair of surface isotopies $\{\phi^t_1\}$ and $\{ \phi^t_2 \}$ with time-one maps $\phi_1$ and $\phi_2$, respectively.
Gambaudo and Ghys proved that its helicity depends linearly on the Calabi invariants of $\phi_1$ and $\phi_2$.
In Section~\ref{sec:suspensions} we calculate the following formula.

\begin{theo}\label{thm:double-suspension}
The helicity of the double suspension $X (\{ \phi_1^t \},\{ \phi_2^t \})$ with respect to the standard volume form $dV$ on the three-sphere $S^3$ equals
	\[ \Hel (X (\{ \phi_1^t \},\{ \phi_2^t \})) = 2 \pi^2 \cdot \left( \Cal (\phi_1) + \Cal (\phi_2) + 2 \pi^2 \right). \]
\end{theo}

In both the case of a vector field preserving a regular contact form, and the case of the suspension of a surface isotopy, the helicity is an invariant of the generating Hamiltonian \emph{function} rather than the vector field or isotopy.
This allows an extension of the invariant to isotopies of volume preserving homeomorphisms, and to show that conjugation by volume preserving homeomorphisms does not alter the helicity, provided the isotopies and homeomorphisms can be described as lifts from a surface.
This is explained in greater detail below.
As an immediate corollary to Theorem~\ref{thm:helicity} we have the following result.

\begin{cor} \label{cor:helicity}
Let $M$ be a closed three-manifold with a regular contact form $\alpha$, and $H_i$ a uniformly Cauchy sequence of smooth basic functions on $M$, such that the corresponding strictly contact isotopies $\{ \phi_{H_i}^t \}$ are uniformly Cauchy as well.
Then the sequence $\Hel (X_{H_i})$ of real numbers converges.
\end{cor}

\begin{dfn}\label{def:helicity}
Let $M$ be a closed three-manifold with a regular contact form $\alpha$.
The \emph{helicity} of a \emph{continuous strictly contact isotopy} $\{ \phi_t \}$ with unique continuous contact Hamiltonian function $H$ is defined to be
	\[ \Hel (\{ \phi_t \}) = \lim_{i \to \infty} \Hel (X_{H_i}) = \left( 4 c^2 (H) - 3 c (H^2) \right) \cdot \vol (M,\alpha \wedge d\alpha), \]
where $H_i$ is a uniformly Cauchy sequence of smooth basic functions with $\{ \phi_{H_i}^t \} \to \{ \phi_t \}$ uniformly.
\end{dfn}

The definition of a continuous strictly contact isotopy \cite{banyaga:ugh11} is given below in Section~\ref{sec:topo-contact}, where we also prove the following result.

\begin{theo} \label{thm:strictly-contact-homeo-conjugation}
Let $M$ be a closed three-manifold with a regular contact form $\alpha$.
If a homeomorphism $\phi$ of $M$ is the uniform limit of a sequence of strictly contact diffeomorphisms, then for any two smooth or continuous strictly contact isotopies $\{ \phi_t \}$ and $\{ \psi_t \} = \{ \phi \circ \phi_t \circ \phi^{-1} \}$ that are conjugated by $\phi$, we have the identity $\Hel (\{ \phi_t \}) = \Hel (\{ \psi_t \})$.
\end{theo}

Regarding suspensions of continuous surface isotopies, Theorems~\ref{thm:suspension} and \ref{thm:double-suspension} imply the following corollary.

\begin{cor}\label{cor:helicity-twice-calabi}
Suppose $H_i$ is a uniformly Cauchy sequence of smooth Hamiltonian functions on the disk $D^2$, and the corresponding Hamiltonian isotopies $\{ \phi_{H_i}^t \}$ converge uniformly to a continuous isotopy $\{ \phi_t \}$.
Then the helicities $\Hel (\tau_* X (\{ \phi_{H_i}^t \}))$ with respect to the standard volume form $dV$ on $\R^3$ of the suspensions $\tau_* X (\{ \phi_{H_i}^t \})$ converge to twice the Calabi invariant of $\{ \phi_t \}$ with respect to the area form $\omega$.

Moreover, if $F_i$ is another uniformly Cauchy sequence of smooth Hamiltonian functions, and the corresponding Hamiltonian isotopies $\{ \phi_{F_i}^t \}$ converge uniformly to a continuous isotopy $\{ \psi_t \}$, then the helicities $\Hel (X (\{ \phi_{H_i}^t \}, \{ \phi_{F_i}^t \}))$ with respect to the standard volume form $dV$ on $S^3$ of the double suspensions $X (\{ \phi_{H_i}^t \}, \{ \phi_{F_i}^t \})$ converge to the number $2 \pi^2 \cdot ( \Cal (\{ \phi_t \}) + \Cal (\{ \psi_t \}) + 2 \pi^2 )$.
\end{cor}

\begin{dfn}\label{def:helicity-twice-calabi}
The \emph{helicity} of the \emph{suspension of a continuous Hamiltonian isotopy} $\{ \phi_t \}$ of $(D^2,\partial D^2,\omega)$ is by definition equal to twice the Calabi invariant of $\{ \phi_t \}$.
The \emph{helicity} of the \emph{double suspension of two continuous Hamiltonian isotopies} $\{ \phi_t \}$ and $\{ \psi_t \}$ of $(D^2,\partial D^2,\omega)$ is by definition the number $2 \pi^2 \cdot ( \Cal (\{ \phi_t \}) + \Cal (\{ \psi_t \}) + 2 \pi^2 )$.
\end{dfn}

See Section~\ref{sec:topo-ham} for the definition of continuous Hamiltonian isotopies and their Calabi invariant.

\begin{theo}\label{theo:suspension-conjugation}
Suppose the suspensions of two smooth or continuous Hamiltonian isotopies $\{ \phi_t \}$ and $\{ \psi_t \}$ of $(D^2,\partial D^2,\omega)$ are conjugated by a homeomorphism of the form $(x,t) \mapsto (\varphi (x),t)$, where $\varphi$ is an area preserving homeomorphism of $D^2$ that is the identity near the boundary of the disk.
Then their helicities necessarily coincide.
The same holds for topologically conjugate double suspensions, provided the conjugating homeomorphism of $S^3$ is of the above product form on the two solid tori that in Hopf coordinates are given by $\{ \eta \le \pi / 4 \}$ and $\{ \eta \ge \pi / 4 \}$.
\end{theo}

As a motivation for studying the helicity, we mention the following interesting problem in hydrodynamics, and refer to \cite{arnold:ahi86, arnold:tmh98} for details.
The mathematical model for fluid dynamics is the hydrodynamics of an incompressible inviscid homogeneous fluid filling $M$, or in other words, the (volume preserving) flow of a divergence-free vector field $X$ on $M$.
Let $g$ be some auxiliary Riemannian metric, and define the (magnetic) energy of $X$ with respect to $g$ by $E (X) = \int_M g (X,X)$.
The group $\Diff (M,\mu)$ of volume preserving diffeomorphisms acts on the Lie algebra of divergence-free vector fields by $X \mapsto \phi_* X$.
Consider the problem of minimizing the functional $E$ on the (adjoint) orbit $\left\{ \phi_* X \mid \phi \in \Diff (M,\mu) \right\}$ of a fixed vector field $X$.
For general $X$ there need not be a minimizing (smooth) vector field.
If there is not, can the energy be made arbitrarily small?
For generic $X$, the answer is no.
Arnold~\cite{arnold:ahi86} showed that
\begin{equation} \label{eqn:energy-helicity-ineq}
	E (\phi_* X) \ge C \cdot | \Hel (X) |,
\end{equation}
where $C$ is some positive constant that depends on the metric $g$.
The helicity is invariant under the action of volume preserving diffeomorphisms, and independent of the metric $g$.
For generic $X$, the helicity does not vanish, and the above inequality gives a lower bound for the magnetic energy on the orbit of $X$.
Arnold also proved that the critical points of $E$ restricted to a fixed orbit are precisely those divergence-free vector fields that commute with their curl, including in particular Beltrami fields, i.e.\ eigenfields of the curl operator.
The Hopf field on the three-sphere is an example, cf.\ Section~\ref{sec:hopf-bundle}.
Beltrami fields with respect to some Riemannian metric are Reeb vector fields of some contact form, and vice versa \cite{etnyre:cth00}.
We will review contact geometry in Section~\ref{sec:contact}.
Similar problems in hydrodynamics are discussed in the book by Arnold and Khesin.

Regarding Theorem~\ref{thm:helicity}, we point out that strictly contact vector fields are generalizations of the aforementioned Reeb vector fields, and appear quite naturally in the present context.
A vector field is strictly contact if and only if it is divergence-free and contact, and strictly contact vector fields are those vector fields that commute with the Reeb vector field.
If the contact form is regular, that is, the Reeb vector field induces a free $S^1$-action on $M$, then strictly contact vector fields are precisely the lifts of Hamiltonian vector fields on the quotient of $M$ by the Reeb flow.
Similarly in Theorem~\ref{thm:suspension} and Theorem~\ref{thm:double-suspension}, we consider lifts of isotopies of the disk to the solid two-torus or the three-sphere.

In the first part of the paper, our methods are elementary, and use the calculus of differential forms and the geometry of (regular) contact and symplectic manifolds.
Section~\ref{sec:helicity} reviews the definition of helicity and establishes its most important basic properties.
In Section~\ref{sec:contact} we review the contact geometry of (regular) contact manifolds, and in Section~\ref{sec:helicity-contact} the proof of Theorem~\ref{thm:helicity} is given.
Section~\ref{sec:hopf-bundle} discusses the case of the three-sphere which is of greatest interest.
In Section~\ref{sec:homotopy} homotopies rel end points are considered, and Section~\ref{sec:suspensions} is concerned with suspensions of surface isotopies.

The second part of the paper comprises Section~\ref{sec:topo-contact} and Section~\ref{sec:topo-ham}.
We use tools from topological Hamiltonian and contact dynamics \cite{mueller:ghh07, mueller:ghl08, mueller:ghc08, viterbo:ugh06, buhovsky:ugh11, banyaga:ugh11, mueller:gcd11} to address Arnold's questions.

In the last part of the paper, Sections~\ref{sec:conjugate-diffeomorphisms} and~\ref{sec:conjugate-systems} take up the question of the existence of diffeomorphisms and vector fields that are topologically conjugate but not $C^1$-conjugate.
The proofs use the uniqueness theorems and the transformation laws of topological Hamiltonian and contact dynamics.
Section~\ref{sec:higher-dim} is devoted to higher-dimensional helicities.

In the two appendices, we prove a proposition from Section~\ref{sec:homotopy}, and compute the helicity of strictly contact vector fields on the three-torus.

\section{Helicity of divergence-free vector fields} \label{sec:helicity}
Let $M$ be a closed smooth three-manifold equipped with a volume form $\mu$.
For the time being, assume $H^2 (M) = 0$.
By Cartan's formula, if a smooth vector field $X$ on $M$ is divergence-free, i.e.\ the Lie derivative satisfies $\mathcal L_X \mu = 0$, then the two-form $\iota_X \mu$ is closed, where $\iota$ denotes interior multiplication of a differential form by a vector field.
By our hypothesis, there exists a one-form $\beta = \beta_X$ with $d\beta_X = \iota_X \mu$, called a \emph{primitive} of the two-form $\iota_X \mu$.
The helicity of $X$ is defined to be the real number
\begin{equation} \label{eqn:def-helicity}
	\Hel (X) = \int_M \beta_X \wedge d\beta_X = \int_M \beta_X (X) \cdot \mu.
\end{equation}
This definition does not depend on the choice of primitive $\beta$ of $\iota_X \mu$.
Indeed, suppose $\beta'$ is another one-form satisfying $d\beta' = \iota_X \mu = d\beta$.
Then $\beta - \beta'$ is closed, and we have
	\[ \int_M \beta \wedge d\beta - \int_M \beta' \wedge d\beta' = \int_M (\beta - \beta') \wedge d\beta = \int_M d (\beta \wedge (\beta - \beta')) = 0 \]
by Stokes' theorem.
For example, one can chose $\beta_X = G \delta (\iota_X \mu)$ using the Hodge decomposition with respect to some auxiliary Riemannian metric.
The second equality in (\ref{eqn:def-helicity}) follows from the fact that interior multiplication is an anti-derivation, and $\beta \wedge \mu$ vanishes for dimension reasons.
For later reference, we formalize this argument in the following obvious lemma.

\begin{lemma} \label{lem:differential}
Given a $p$-form $\sigma$ and a $q$-form $\tau$ on a smooth manifold $M$, the $(p + q + 1)$-forms $d \sigma \wedge \tau$ and $\sigma \wedge d\tau$ coincide up to sign and an exact form.
More precisely, $[ d \sigma \wedge \tau ] = (-1)^{p + 1} [ \sigma \wedge d\tau ]$.
In particular, if $M$ is closed, and $p + q = \dim M - 1$, then
	\[ \int_M d \sigma \wedge \tau = (-1)^{p + 1} \int_M \sigma \wedge d\tau. \]
If $p + q > \dim M$, then $\iota_X \sigma \wedge \tau = (-1)^{p + 1} \sigma \wedge \iota_X \tau$ for any vector field $X$.
\end{lemma}

If $H^2 (M)$ is nonzero, the helicity invariant is defined on the Lie subalgebra of divergence-free vector fields $X$ such that $\iota_X \mu$ is exact.
Such vector fields are sometimes called \emph{exact} in the literature.
There is a homomorphism on the Lie algebra of divergence-free vector fields into the $(\dim M - 1)^{\text{st}}$ cohomology group of $M$, defined by $X \mapsto [\iota_X \mu]$ (the \emph{flux} of $X$), and its kernel consists precisely of the exact vector fields.
We refer to \cite{banyaga:scd97} for more on this important homomorphism.
It is shown in \cite{mueller:fsc11} that if the volume form is induced by a regular contact form, this kernel contains all (divergence-free) contact vector fields.

The helicity is a quadratic form on the space of exact divergence-free vector fields.
For $X$ and $Y$ exact, define the \emph{relative helicity}
	\[ \Rel (X,Y) = \int_M \beta_X \wedge d\beta_Y = \int_M \beta_Y \wedge d\beta_X \]
independently of the choices of $\beta_X$ and $\beta_Y$ by Lemma~\ref{lem:differential}.
$\Rel$ is symmetric and $\R$-bilinear, and we have the obvious identities $\Hel (X) = \Rel (X, X)$ and 
\begin{equation}\label{eqn:sum}
\Hel (X \pm Y) = \Hel (X) \pm 2 \Rel (X, Y) + \Hel (Y).
\end{equation}
In particular,
	\[ \left. \frac{d}{d\epsilon} \right |_{\epsilon = 0} \Hel (X + \epsilon Y) = 2 \Rel (X,Y), \]
or $d\Hel (X) = 2 \Rel (X,\cdot)$, and for any nonzero $X$ there exists an exact divergence-free vector field $Y$ such that $\Rel (X,Y)$ is nonzero.
Thus the helicity of a $C^1$-generic (exact divergence-free) vector field does not vanish.

If we want to emphasize the dependence on the volume form $\mu$, we write $\Hel (X;\mu)$, and denote the bilinear form by $\Rel (X,Y;\mu)$.
However, the definitions depend on the choice of volume form on $M$ only up to scaling and a volume preserving change of coordinates.
Recall that by Moser's argument, two volume forms $\mu$ and $\nu$ on $M$ are isotopic if and only if the total volumes of $M$ with respect to $\mu$ and $\nu$ coincide.
Thus up to scaling by a nonzero constant $c$, $\mu$ is isotopic to $\nu$.
That means there exists a diffeomorphism $\phi$ (which is isotopic to the identity) such that $\phi^* \nu = c \mu$.

\begin{lemma} \label{lem:change-vol-form}
If $\mu$ is a volume form on $M$, $\phi$ an orientation preserving diffeomorphism, and $X$ an exact divergence-free vector field with respect to the volume form $\phi^* \mu$, then $\phi_* X$ is exact divergence-free with respect to $\mu$, and $\Hel (\phi_* X;\mu) = \Hel (X;\phi^* \mu)$.
If $c$ is a nonzero constant, then we have $\Hel (X;c \mu) = c^2 \Hel (X;\mu)$.
More generally, if $f$ is a nowhere vanishing smooth function on $M$, and $X$ an exact divergence-free vector field with respect to the volume form $f \mu$, then the vector field $f X$ is exact divergence-free with respect to $\mu$, and the identity $\Hel (X;f \mu) = \Hel (f X;\mu)$ holds.
Analogous statements hold for the relative helicity $\Rel (X,Y)$.
In particular, both $\Hel$ and $\Rel$ are invariant under the action of volume preserving diffeomorphisms on (exact) divergence-free vector fields.
\end{lemma}

We note that the flow of $\phi_* X$ is the conjugation $\phi \circ \phi_X^t \circ \phi^{-1}$ of the flow $\phi_X^t$ of $X$ by $\phi$, and the flow of $f X$ is related to the flow of $X$ by the formula
	\[ \phi_{f X}^t (x) = \phi_X^{\tau (t,x)} (x). \]
Here the smooth function $\tau \colon \R \times M \to \R$ solves the following ordinary differential equation with initial condition $\tau (0,x) = 0$ for all $x \in M$:
	\[ \frac{d}{dt} \tau (t,x) = f (\phi_X^{\tau (t,x)}). \]

\begin{proof}
It is straightforward to check the well-known identity
\begin{equation} \label{eqn:change-vol-form}
	\phi^* (\iota_{\phi_* X} \mu) = \iota_X (\phi^* \mu).
\end{equation}
Therefore if $\beta_X$ is a primitive of $\iota_X \phi^* \mu$, then $(\phi^{-1})^* \beta_X$ is a primitive of $\iota_{\phi_* X} \mu$.
The first claim now follows from the change of variables formula.
The other identities are proved similarly.
\end{proof}

We may consider the helicity as an invariant of the volume preserving isotopy $\{ \phi_X^t \}_{0 \le t \le 1}$ generated by the vector field $X$, i.e.\ ${d / dt} \, \phi_X^t = X \circ \phi_X^t$, and $\phi_X^0$ is the identity.
There is also a flux homomorphism defined for volume preserving isotopies \cite{banyaga:scd97}, and in fact, the flux of $\{ \phi_X^t \}$ by definition equals the flux of its infinitesimal generator $X$.
Thus if $H^2 (M) \not= 0$, the helicity is defined for exact volume preserving isotopies, i.e.\ those in the kernel of the flux map.

By Hodge theory, the helicity depends continuously on the vector field $X$, provided we equip the Lie algebra of divergence-free vector fields with the $C^1$-topology.
As we previously noted, with respect to the $C^0$-topology, the helicity does not depend continuously on the isotopy $\{ \phi_X^t \}$ generated by $X$.

The helicity is also defined for $M$ compact and connected with nonempty boundary, provided $M$ is simply-connected, and $X$ is tangent to the boundary of $M$.
One can also define the helicity for compact connected embedded submanifolds of $\R^3$ with nonempty boundary, if the divergence-free (with respect to the standard volume form $dV$ on $\R^3$) vector field $X$ is again tangent to the boundary.
In the latter case however the helicity does depend on the embedding into $\R^3$, see Section~\ref{sec:suspensions}.
We refer to \cite{arnold:tmh98, gambaudo:ea97} for details.

The definition of $\Hel (X)$ generalizes in an obvious fashion to time-dependent vector fields $\{ X_t \} $.
Suppose the closed two-forms $\iota_{X_t} \mu$ are exact for all $0 \le t \le 1$.
By Hodge theory, after choosing an auxiliary Riemannian metric on $M$, we may choose the primitives $\beta_t$ satisfying $d\beta_t = \iota_{X_t} \mu$ to depend smoothly on $t$.
Then define
	\[ \Hel (\{ X_t \}) = \int_0^1 \int_M \beta_t \wedge d\beta_t \, dt. \]
This number is again well-defined, and coincides with the previous definition if $X$ is autonomous.
We can also define the helicity if only the time average of $\iota_{X_t} \mu$ is exact, but the forms $\iota_{X_t} \mu$ are not necessarily exact for all times.
This definition also extends the definition for autonomous $X$, but the two definitions for time-dependent vector fields may not coincide if $\iota_{X_t} \mu$ happens to be exact for all $t$.
We remark that in the second situation the flow $\{ \phi_X^t\}$ of $X_t$ is isotopic rel end points to an exact isotopy $\{ \phi_Y^t\}$ \cite{banyaga:scd97}.
However, the helicity does in general depend on the homotopy class (rel end points) of the isotopy, see Section~\ref{sec:homotopy}.

The classical Hopf invariant of the homotopy class of a map $p \colon S^3 \to S^2$ is defined as follows.
Choose an area form $\omega$ of total area $1$ on $S^2$, and a primitive $\beta$ of the (closed and hence exact) two-form $p^* \omega$ on $S^3$.
Then define the Hopf invariant as the integral $\int \beta \wedge d\beta$ over $S^3$.
This is an integer which is also equal to the linking number of the preimages under $p$ of two regular points in $S^2$.
By the nondegeneracy of $\mu$, every closed (exact) two-form on an oriented three-manifold can be written $\iota_X \mu$ for some divergence-free (exact) vector field.
The generalized Hopf invariant is defined even if the two-form $\iota_X \mu$ is not the pull-back of a closed form on $S^2$, and can take any real value.
In order to prove Theorem~\ref{thm:helicity}, we will consider the projection $p \colon M \to B$ of the {\em Boothby-Wang} (or {\em prequantization}) bundle over an integral symplectic surface, and relate the form $\iota_X \mu$ to the pull-back of an exact form on the base $B$.
This set-up will be explained in the next two sections, and a similar strategy will be applied in Section~\ref{sec:suspensions} to prove Theorem~\ref{thm:suspension} and Theorem~\ref{thm:double-suspension}.

\section{Regular contact manifolds} \label{sec:contact}
Let $M$ be a closed smooth manifold of dimension ${2 n + 1}$, equipped with a coorientable nowhere integrable field of hyperplanes (a \emph{contact distribution} or \emph{contact structure}) $\xi \subset TM$.
That means we suppose $\xi$ is given (globally) by the kernel $\xi = \ker \alpha$ of a differential one-form $\alpha$, and $\mu = \alpha \wedge (d\alpha)^n$ is a volume form on $M$.
We call $\mu$ the \emph{canonical volume form} induced by the \emph{contact form} $\alpha$.
For readers not familiar with contact (and symplectic) geometry, we recommend the monographs \cite{mcduff:ist98, geiges:ict08}.
For simplicity, we assume throughout this article that $M$ is connected.

A vector field $X$ on $M$ is said to be \emph{contact} (with respect to $\xi$) if $\mathcal L_X \alpha = h_X \alpha$ for a smooth function $h_X$ on $M$, and \emph{strictly contact} (with respect to $\alpha$) if $h_X = 0$.
Hence, $X$ is contact if and only if its flow $\phi_X^t$ preserves the contact structure $\xi$, and strictly contact if and only if its flow preserves the contact form $\alpha$.
Note that the former concept depends only on the contact structure $\xi$, whereas the latter concept depends on the actual choice of contact form $\alpha$.
A vector field on $M$ is divergence-free and contact if and only if it is strictly contact.
For any $f \in C^\infty (M)$, the one-form $e^f \alpha$ defines another contact form giving rise to the same coorientation of $\xi$ and orientation of $M$, and all contact forms representing the cooriented contact structure $\xi$ can be written in this way.

We denote by $R_\alpha$ the \emph{Reeb vector field} of the contact form $\alpha$, i.e.\ the unique smooth vector field defined by the equations $\iota_{R_\alpha} d\alpha = 0$ and $\iota_{R_\alpha} \alpha = 1$, and call its flow the Reeb flow on $(M,\alpha)$.
More generally, given a contact vector field $X$, we call the smooth function $H = \iota_X \alpha$ its {\em contact Hamiltonian}.
Conversely, given a smooth function $H$ on $M$, there is a unique contact vector field $X$ with contact Hamiltonian $H$ and satisfying the equation $\iota_X d\alpha = (R_\alpha . H) \alpha - d H$.
Here we write $X . f = df (X)$ for the derivative of a smooth function $f$ in the direction of a vector field $X$ on $M$.
Indeed, $d\alpha$ restricted to the subbundle $\xi \subset TM$ is nondegenerate (i.e.\ $(\xi,d\alpha|_\xi)$ is a symplectic vector bundle over $M$), so that the two equations together uniquely define the contact vector field $X$.
We write $X_H$ for the contact vector field with contact Hamiltonian $H$, and denote its flow by $\Phi_H = \{ \phi_H^t \}$.
Observe that $h_X = R_\alpha . H$, so that $X$ is strictly contact if and only if $R_\alpha . H = 0$, or equivalently, $H$ is preserved under the flow of $R_\alpha$.
Such functions are called \emph{basic} functions.

The contact form $\alpha$ is said to be {\em regular} if $R_\alpha$ generates a free $S^1$-action on $M$; in particular, all Reeb orbits are closed and of period $1$, and $M$ is the total space of a principle $S^1$-bundle known as the \emph{Boothby-Wang} bundle \cite{boothby:cm58}
\begin{equation} \label{eqn:bw-bundle}
	S^1 \stackrel{i}{\longrightarrow} M \stackrel{p}{\longrightarrow} B
\end{equation}
over a closed and connected integral symplectic manifold $(B,\omega)$, such that $p^* \omega = d\alpha$.
Recall that \emph{symplectic} means that the two-form $\omega$ on $B$ is closed and nondegenerate, i.e.\ its top power $\omega^n$ defines a volume form on $B$, and integral means that the cohomology class $[\omega] \in H^2 (M,\Z)$.
Here $i$ is the $S^1$-action of the Reeb vector field, and $p \colon M \to B \cong M / S^1$ is the projection to the quotient.
The projection induces an (algebra) isomorphism $p^* \colon C^\infty (B) \to C^\infty_b (M)$ between smooth functions on the base $B$, and smooth basic functions on $M$, and a surjective homomorphism $p_* (X_H) = - X_F$ between strictly contact vector fields $X_H$ on $(M,\alpha)$, and Hamiltonian vector fields $X_F$ on $(B,\omega)$ (with kernel generated over $\R$ by $R_\alpha$).
Here $H = p^* F = F \circ p$, and the vector field $X_F$ is uniquely defined by the equation $\iota_{X_F} \omega = dF$.

If $M$ has dimension $3$, then $B = \Sigma_g$ is an oriented closed and connected surface of genus $g$ with integral total area.
The most interesting case is when the genus $g$ is zero.
Then (\ref{eqn:bw-bundle}) is the Hopf bundle $S^1 \to S^3 \to S^2$ (Section~\ref{sec:hopf-bundle}).
We would like to point out that every closed orientable three-manifold admits a contact structure \cite{martinet:fcv71}, but none of its contact forms need be regular.
For example, the three-torus $T^3$ does not admit a regular contact form (and in fact, no torus $T^{2n + 1}$ does) \cite{blair:rgc10}.
This case is discussed separately in Appendix~\ref{sec:torus}.

Banyaga \cite{banyaga:gdp78} has shown that the Boothby-Wang bundle (\ref{eqn:bw-bundle}) gives rise to a short exact sequence
\begin{equation} \label{eqn:s^1-ext}
	1 \longrightarrow S^1 \stackrel{i_*}{\longrightarrow} \Diff_0 (M,\alpha) \stackrel{p_*}{\longrightarrow} \Ham (B,\omega) \longrightarrow 1
\end{equation}
with $S^1$ in the center of $\Diff_0 (M,\alpha)$.
Here $\Diff_0 (M,\alpha)$ denotes the group of \emph{strictly contact diffeomorphisms}, i.e.\ all those diffeomorphisms preserving the contact form $\alpha$ and isotopic to the identity through an isotopy of diffeomorphisms preserving $\alpha$, $\Ham (B,\omega)$ denotes the group of \emph{Hamiltonian diffeomorphism} of $(B,\omega)$, that is, time-one maps of (the isotopies generated by time-dependent) Hamiltonian vector fields, and $i_*$ is again the $S^1$-action of the Reeb vector field.

Note that $\iota_{R_\alpha} (\alpha \wedge d\alpha) = d\alpha$, so that $\Hel (R_\alpha) = \int_M \alpha \wedge d\alpha = \vol (M)$.
Theorem~\ref{thm:helicity} generalizes this computation to all strictly contact vector fields on a regular contact manifold $M$.

\section{Helicity of strictly contact vector fields} \label{sec:helicity-contact}
Let $M$ be a smooth manifold with a volume form $\mu$, and define a (group) homomorphism $c \colon C^0 (M) \to \R$ by
\begin{equation} \label{eqn:average}
	H \mapsto c (H) = c_H = \frac{1}{\vol (M,\mu)} \int_M H \, \mu,
\end{equation}
where $\vol (M,\mu) = \int_M \mu$ is the total volume of $M$ with respect to $\mu$.
The next lemma shows that if $S^1 \to M \to B$ is the Boothby-Wang bundle (\ref{eqn:bw-bundle}) over an integral symplectic manifold $(B^{2n},\omega)$, then the projection $p$ preserves the homomorphism $c$.
In other words, $c_B = c_M \circ p^*$.
Here $c_M$ denotes the average value (\ref{eqn:average}) with respect to the canonical volume form $\alpha \wedge (d\alpha)^n$ on the total space $M$, and similarly $c_B$ denotes the average value (\ref{eqn:average}) with respect to the canonical volume form $\omega^n$ on the base $B$.

\begin{prop} \label{pro:average}
Let $(M^{2n + 1},\alpha)$ be a regular contact manifold, and write $H = p^* F = F \circ p$ for $F \in C^0 (B)$, where $p$ is the projection map of the Boothby-Wang bundle (\ref{eqn:bw-bundle}).
Then
	\[ \int_M H \, \alpha \wedge (d\alpha)^n = \int_B F \, \omega^n. \]
In particular, we have $c_M (H) = c_B (F)$ with respect to the canonical volume forms $\alpha \wedge (d\alpha)^n$ and $\omega^n$.
\end{prop}

\begin{proof}
Choose an open cover $\{ U_i \}$ of $B$ with the property that the bundle is trivial over each $U_i$, and let $\{ \lambda_i \}$ be a partition of unity subordinate to $\{ U_i \}$.
Denote by $V_i = p^{-1} (U_i) \cong U_i \times S^1$, and by $\mu_i = p^* \lambda_i = \lambda_i \circ p$ the partition of unity subordinate to the open cover $\{ V_i \}$ of $M$.
We compute
\begin{eqnarray*}
	\int_M H \, \alpha \wedge (d\alpha)^n &=& \sum_i \int_{U_i \times S^1} (\mu_i \, H)\, \alpha \wedge (d\alpha)^n \\
	&=& \sum_i \int_{U_i \times \{ \text{pt} \}} (\mu_i \, H) (d\alpha)^n \\
	&=& \sum_i \int_{U_i} (\lambda_i \, F) \, \omega^n \\
	&=& \int_B F \, \omega^n,
\end{eqnarray*}
which is what we set out to prove.
The second equality follows from the fact that $\mu_i$ and $H$ are constant on the Reeb orbits $S^1$.
Applying the above formula to the constant function $1$ proves the last part of the proposition.
\end{proof} 

\begin{proof}[First proof of Theorem~\ref{thm:helicity}]
The arguments in the proof are valid in any (odd) dimension, except the definition of helicity only makes sense when $\dim M = 3$.
Thus consider the Boothby-Wang bundle $S^1 \rightarrow M^{2n + 1} \rightarrow B^{2n}$ over an integral $(B,\omega)$.
Denote by $F \in C^\infty (B)$ the unique function such that $H = p^* F$, and write
	\[ c_F = \frac{1}{\vol (B,\omega^n)} \int_B F \, \omega^n, \]
for the average value of $F$ with respect to the canonical volume form $\omega^n$ on the base $B$.
Then $\int_B (F - c_F) \, \omega^n = 0$, and the $2n$-form $(F - c_F) \, \omega^n$ on $B$ is exact.
Choose a primitive $\gamma$, i.e.\ a $(2n - 1)$-form such that $d\gamma = (F - c_F) \, \omega^n$.
By Proposition~\ref{pro:average}, we have $c_F = c_H$.
Define a $(2n - 1)$-form $\beta$ on $M$ by
\begin{equation} \label{eqn:beta}
	\beta = (n + 1) p^* \gamma + ((n + 1) c_H  - n H) \alpha \wedge (d\alpha)^{n - 1}.
\end{equation}
By construction,
\begin{align*}
	d\beta &= (n + 1) ( p^* d\gamma + c_H (d\alpha)^n ) - n (dH \wedge \alpha + H d\alpha) \wedge (d\alpha)^{n - 1} \\
	&= (n + 1) ( (H - c_H) (d\alpha)^n + c_H (d\alpha)^n ) - n (dH \wedge \alpha + H d\alpha) \wedge (d\alpha)^{n - 1} \\
	&= H (d\alpha)^n + n \alpha \wedge dH \wedge (d\alpha)^{n - 1} \\
	&= H (d\alpha)^n - \alpha \wedge n (\iota_{X_H} d\alpha) \wedge (d\alpha)^{n - 1} \\
	&= \iota_{X_H} (\alpha \wedge (d\alpha)^n).
\end{align*}
In the case $n = 1$, this becomes
	\[ \beta = 2 p^* \gamma + (2 c_H - H) \alpha, \ {\rm and} \ d\beta = 2 H d\alpha - d (H \alpha) = \iota_{X_H} (\alpha \wedge d\alpha). \]
We obtain
\begin{align*}
	\beta \wedge d\beta &= ( 2 p^* \gamma + (2 c_H - H) \alpha ) \wedge ( 2 H d\alpha - d (H \alpha) ) \\
	&= 4 H p^* \gamma \wedge d\alpha - 2 p^* \gamma \wedge d (H \alpha) + (2 c_H - H) 2 H \alpha \wedge d\alpha \\
	& \hspace{1cm} - \; (2 c_H - H) \alpha \wedge d (H \alpha) \\
	&= 4 H p^* (\gamma \wedge \omega) - 2 H p^* d \gamma \wedge \alpha + \text{exact terms} \\
	& \hspace{1cm} + \; (2 c_H - H) H \alpha \wedge d\alpha \\
	&= - 2 H (H - c_H) d\alpha \wedge \alpha + (2 c_H H - H^2 ) \alpha \wedge d\alpha + \text{exact terms} \\
	&= (4 c_H H - 3 H^2) \alpha \wedge d\alpha + \text{exact terms},
\end{align*}
where we have used that the wedge product is graded commutative, and Lemma~\ref{lem:differential} for the third equality.
Therefore
	\[ \Hel (X_H) = \int_M \beta \wedge d\beta = \left( 4 c^2 (H) - 3 c (H^2) \right) \cdot \vol (M,\alpha \wedge d\alpha).\qedhere \]
\end{proof}

\begin{proof}[Second proof of Theorem~\ref{thm:helicity}]
Alternatively, suppose $\beta$ is given by (\ref{eqn:beta}), then
	\[ \beta (X_H) = (n + 1) (p^* \gamma) (X_H) + ((n + 1) c_H - n H) (\alpha \wedge (d\alpha)^{n - 1}) (X_H). \]
Since $p^* (\iota_{p_* X_H} \omega) = \iota_{X_H} p^* \omega = \iota_{X_H} d\alpha = - dH = - p^* (dF) = p^* (\iota_{- X_F} \omega)$, and $p^*$ is an isomorphism on exact one-forms, $X_H$ indeed has a well-defined projection $p_* X_H = - X_F$.
We see that $(p^* \gamma) (X_H) = p^* (\gamma (- X_F))$ (compare to (\ref{eqn:change-vol-form})), and by the same argument as in the proof of Proposition~\ref{pro:average},
	\[ \int_M (p^* \gamma) (X_H) \wedge \alpha \wedge d\alpha = - \int_B \gamma (X_F) \wedge \omega = - \int_B \gamma \wedge (\iota_{X_F} \omega). \]
The latter coincides with
	\[ - \int_B \gamma \wedge dF = - \int_B F d\gamma = - \int_B F (F - c_F) \omega^n. \]
Recalling that $p^*$ preserves average values by Proposition~\ref{pro:average}, we obtain
	\[ \int_M (p^* \gamma) (X_H) \wedge \alpha \wedge d\alpha = \left( c^2 (H) - c (H^2) \right) \cdot \vol (M,\alpha \wedge (d\alpha)^n). \]
Integrating the above expression for $\beta (X_H)$ over $M$ in the case $n = 1$ completes the proof.
\end{proof}

By Theorem~\ref{thm:helicity}, for volume preserving contact isotopies on regular contact manifolds, the helicity is an invariant of the generating Hamiltonian function rather than the corresponding vector field or isotopy.
By Theorem~\ref{thm:suspension} and Theorem~\ref{thm:double-suspension}, the same holds for suspensions of surface isotopies on the solid two-torus or the three-sphere.

We would like to alert the reader that this formula is only valid for the canonical volume form $\alpha \wedge d\alpha$.
If $\mu = c \, \alpha \wedge d\alpha$, then $\Hel (X;\mu) = c^2 \Hel (X;\alpha \wedge d\alpha)$.
If $M$ admits a regular $\alpha$, unless explicitly stated otherwise, we always assume the volume form is the canonical one induced by the contact form $\alpha$.
If $\alpha$ is not regular, then $X_H$ need not be exact \cite{mueller:fsc11}, see also Appendix~\ref{sec:torus}.

The same argument proves a relative version of Theorem~\ref{thm:helicity}.

\begin{theo}
Let $M$ be a closed three-manifold equipped with a contact form $\alpha$ as in Theorem~\ref{thm:helicity}, and let $X_H$ and $X_K$ be strictly contact vector fields with contact Hamiltonian functions $H$ and $K \in C^\infty_b (M)$.
Then $X_H$ and $X_K$ are exact divergence-free, and
	\[ \Rel (X_H,X_K) = \left( 4 c (H) c (K) - 3 c (H \cdot K) \right) \cdot \vol (M,\alpha \wedge d\alpha). \]
\end{theo}

The map $\Rel (H,K) = (4 c (H) c (K) - 3 c (H \cdot K)) \cdot \vol (M,\alpha \wedge d\alpha)$ is symmetric and $\R$-bilinear, and defines a quadratic form on $C_b^\infty (M)$ given by $\Hel (H) = \Rel (H,H) = (4 c^2 (H) - 3 c (H^2)) \cdot \vol (M,\alpha \wedge d\alpha)$.
By definition, $\Rel (X_H,X_K) = \Rel (H,K)$ and $\Hel (X_H) = \Hel (H)$.
We again see that $\Hel (H \pm K) = \Hel (H) \pm 2 \Rel (H,K) + \Hel (K)$, and thus the helicity of a $C^1$-generic strictly contact vector field is nonzero.

Alternatively, the average value $c_H = c_F$ can be computed as follows.

\begin{lemma}
Let $S \subset M$ be a Reeb circle, i.e.\ the preimage $p^{-1} (b)$ of a point $b \in B$, and $D \subset M$ any filling disc, that is, $\partial D = S$.
Then
	\[ c_H = \int_D H \, d\alpha \]
independently of $b \in B$ and disc $D \subset M$ with boundary $S = p^{-1} (b)$.
\end{lemma}

\begin{proof}
One way to see the above identity is as follows.
First note that
	\[ \int_D H \, d\alpha = \int_D d (p^* \gamma + c_F \alpha) = \int_S p^* \gamma + c_F \alpha. \]
Let $t \mapsto x (t)$ be a parameterization of the Reeb circle by arc length.
Then the last integral is equal to
	\[ \int_0^1 \iota_{R_\alpha} (p^* \gamma + c_F \alpha) ( x (t) ) \, dt = \int_0^1 \iota_{p_* R_\alpha} \gamma ( b ) \, dt + c_F = c_F, \]
since $p_* R_\alpha = 0$.
The lemma now follows from Proposition~\ref{pro:average}.
\end{proof} 

\begin{exa}\label{exa:lift}
We can decompose any strictly contact vector field $X_H$ into its horizontal and vertical parts $(X_H - H R_\alpha) + H R_\alpha$ with respect to the projection $p_*$.
Note that $X_{H + c} = X_H + c R_\alpha$, so the kernel of the (surjective) homomorphism $p_*$ is indeed generated over the reals by $R_\alpha$.
For a smooth function $F \in C^\infty (B)$, the horizontal lift $Y_F = H R _\alpha - X_H$ of the Hamiltonian vector field $X_F$ is exact, and by H{\"o}lder's inequality,
	\[ \Hel (Y_F) = ( c^2 (H) - c (H^2) ) \cdot \vol (M) = (c^2(F) - c (F^2)) \cdot \vol (B) \le 0, \]
with equality if and only if $F$ is constant, or equivalently, $X_F = 0$.
\end{exa}

\begin{prop}
The absolute value of the helicity $\Hel (X_H)$ is bounded by a constant times the square of the $L^2$-norm of $H$ on $(M,\alpha \wedge d\alpha)$.
In fact,
	\[-3 \| H \|_{L^2}^2 \le \Hel (X_H) \le \| H \|_{L^2}^2, \]
with equality if and only if $H$ has mean value zero or is constant, respectively.
Moreover, the restriction of the helicity to strictly contact vector fields can take any real value.
\end{prop}

Compare to Arnold's inequality (\ref{eqn:energy-helicity-ineq}).

\begin{proof}
Consider a basic function $H \in C^\infty_b (M)$, and denote by $\| \cdot \|_{L^2}$ the $L^2$-norm on $(M,\alpha \wedge d\alpha)$, i.e.\ $\| H \|_{L^2}^2 = \int_M H^2 \alpha \wedge d\alpha$.
Again H\"older's inequality implies $c^2 (H) \le c (H^2)$, with equality if and only if $H$ is a constant function.
Thus
\begin{align*}
-3 \| H \|_{L^2}^2 &= -3c(H^2)\cdot \vol (M)\\
&\le \Hel (X_H)\\ 
&= ( 4 c^2 (H) - 3 c (H^2) ) \cdot \vol (M)\\ 
&\le c (H^2)  \cdot \vol (M)\\
&= \| H \|_{L^2}^2.
\end{align*}

A straightforward calculation shows
	\[ \Hel (X_{H - c}) = \Hel (X_H) - 2 c c_H \cdot \vol (M) + c^2 \cdot \vol (M), \]
which is a quadratic function of $c \in \R$ with global minimum at $c = c_H$.
Thus if $H$ has negative mean value, the helicity takes any real value on vector fields of the form $c_1 X_{H - c_2}$ for $c_1, c_2 \in \R$.
\end{proof}

\section{The Hopf bundle} \label{sec:hopf-bundle}
On the unit three-sphere $S^3 \subset \C^2$, write
\begin{align*}
& z_1 = x_1 + i y_1 = e^{i \xi_1} \sin \eta = \cos \xi_1 \sin \eta + i \sin \xi_1 \sin \eta, \\
& z_2 = x_2 + i y_2 = e^{i \xi_2} \cos \eta = \cos \xi_2 \cos \eta + i \sin \xi_2 \cos \eta,
\end{align*}
for $(z_1,z_2) \in S^3$, where $0 \le \eta \le \frac{\pi}{2}$, and $0 \le \xi_1, \xi_2 < 2 \pi$ are \emph{Hopf coordinates}.
The standard regular contact form on $S^3$ is
	\[ \alpha = \frac{1}{2 \pi} ( x_1 dy_1 - y_1 dx_1 + x_2 dy_2 - y_2 dx_2 ) = \frac{1}{2 \pi} ( \sin^2 \eta \, d\xi_1 + \cos^2 \eta \, d\xi_2 ), \]
with
	\[ d\alpha = \frac{1}{\pi} ( dx_1 \wedge dy_1 + dx_2 \wedge dy_2 ) = \frac{1}{2 \pi} \sin (2 \eta) d\eta \wedge ( d\xi_1 - d\xi_2 ), \]
so that the Reeb vector field is equal to the Hopf vector field given by $R_\alpha = 2 \pi ( \frac{\partial}{\partial \xi_1} + \frac{\partial}{\partial \xi_2} )$, which generates the (one-periodic) Reeb or Hopf flow on $S^3$.
The corresponding volume form is
	\[ \alpha \wedge d\alpha = \frac{1}{(2 \pi)^2} \sin (2 \eta) d\eta \wedge d\xi_1 \wedge d\xi_2 = \frac{1}{2 \pi^2} dV, \]
where $dV$ is the standard volume form on the unit three-sphere.
The total volume of $S^3$ with respect to $\alpha \wedge d\alpha$ equals $1$.

On the unit two-sphere $S^2 \subset \R^3$, consider \emph{spherical coordinates}
	\[ x = \cos \varphi, \ \ \ y = \sin \varphi \cos \psi, \ \ \ z = \sin \varphi \sin \psi, \]
where $0 \le \varphi \le \pi$, and $0 \le \psi < 2 \pi$.
The standard area (or symplectic) form is (up to scaling)
	\[ \omega = \frac{1}{4 \pi} (x\, dy \wedge dz + y\, dz \wedge dx + z\, dx \wedge dy ) = \frac{1}{4 \pi} \sin \varphi\, d\varphi \wedge d\psi = \frac{1}{4 \pi} dV, \]
where again $dV$ denotes the standard area form on the unit two-sphere.
This gives $S^2$ a total area of $1$ with respect to $\omega$.

Recall the Hopf bundle $S^1 \stackrel{i}{\longrightarrow} S^3 \stackrel{p}{\longrightarrow} S^2$.
In the above coordinates, the projection becomes $p (\eta, \xi_1, \xi_2) = (2 \eta, \xi_1 - \xi_2) = (\varphi, \psi)$.
We have
	\[ p^* \omega = \frac{1}{4 \pi} \sin (2 \eta) d (2 \eta) \wedge d (\xi_1 - \xi_2) = \frac{1}{2 \pi} \sin (2 \eta) d\eta \wedge (d\xi_1 - d\xi_2) = d\alpha, \]
so that $p \colon (S^3,\alpha) \to (S^2,\omega)$ is the prequantization bundle (\ref{eqn:bw-bundle}) over the integral symplectic surface $(S^2,\omega)$.
By Theorem~\ref{thm:helicity} and Proposition~\ref{pro:average},
	\[ \Hel (X_H) = 4 c_H^2 - 3 c_{H^2} = 4 c_F^2 - 3 c_{F^2}, \]
where $H = F \circ p$, and where the average values are computed with respect to the volume form $\alpha \wedge d\alpha$ on $S^3$ and the area form $\omega$ on $S^2$.

\begin{exa}\label{exa:sphere}
The Reeb vector field $R_\alpha = X_H$ with $H = 1$ generates the Reeb flow on $S^3$, and $\Hel (R_\alpha) = 4 c (1)^2 - 3 c (1^2) = 1$.
This vector field (as well as its negative) is an eigenvector with eigenvalue $1$ of the curl, and an energy minimizer on its adjoint orbit, with respect to an associated Riemannian metric $g = \alpha \otimes \alpha + d\alpha (\cdot, J \cdot)$ \cite{blair:rgc10}.

Let $X_H = 2 \pi ( \frac{\partial}{\partial \xi_2} - \frac{\partial}{\partial \xi_1} )$, then $H = \cos (2 \eta)$, or $F = \cos \varphi$ with $H = F \circ p$.
	\[ c_F = \frac{1}{4 \pi} \int_0^\pi \int_0^{2 \pi} \cos \varphi \sin \varphi \, d\varphi \, d\psi = 0, \]
and
	\[ c_{F^2} = \frac{1}{4 \pi} \int_0^\pi \int_0^{2 \pi} \cos^2 \varphi \sin \varphi \, d\varphi \, d\psi = \frac{1}{3}, \]
therefore $\Hel (X_H) = - 1$.

Consider $H = \cos^2 \eta = \frac{1}{2} (1 + \cos (2 \eta))$ and $\sin^2 \eta = \frac{1}{2} (1 -  \cos (2\eta))$, corresponding to the strictly contact vector fields $X_H = 2 \pi \frac{\partial}{\partial \xi_2}$ and $2 \pi \frac{\partial}{\partial \xi_1}$.
Then $F = \frac{1}{2} (1 + \cos \varphi)$ and $\frac{1}{2} (1 - \cos \varphi)$, respectively, and we compute as above $\Hel (X_H) = 0$ in both cases.

By Equation~(\ref{eqn:sum}), we can compute the relative helicity of these vector fields.
For example, $4 \pi \frac{\partial}{\partial \xi_2} = R_\alpha + 2 \pi (\frac{\partial}{\partial \xi_2} - \frac{\partial}{\partial \xi_1})$, and thus $\Rel (R_\alpha, 2 \pi (\frac{\partial}{\partial \xi_2} - \frac{\partial}{\partial \xi_1})) = 0$.
\end{exa}

\section{Homotopies rel end points} \label{sec:homotopy}
We begin by recalling the following proposition, which is essentially contained in \cite{banyaga:gdp78}.
For the readers' convenience, a complete proof is given in Appendix~\ref{sec:proof}.

\begin{prop} \label{pro:homotopy-groups}
Let $(M,\alpha)$ be a closed and connected regular contact three-manifold, and $S^1 \to M \to B$ be the associated Boothby-Wang bundle (\ref{eqn:bw-bundle}) over the closed and connected integral symplectic surface $(B,\omega)$.
If the base $B$ has positive genus, then the inclusion $S^1 \hookrightarrow \Diff_0 (M,\alpha)$ into the identity component of the group of strictly contact diffeomorphisms is a homotopy equivalence.
In particular, the fundamental group of $\Diff (M,\alpha)$ is $\Z$, with generator the homotopy class of the one-periodic Reeb flow, and for $k > 1$, $\pi_k ( \Diff (M,\alpha) )$ is trivial.
If the base $B = S^2$, i.e.\ the Boothby-Wang bundle is the Hopf fibration, then the one-periodic Reeb flow represents twice the generator of $\pi_1 ( \Diff (S^3,\alpha) ) = \Z$, and the fundamental group is generated by the homotopy class of the flow of the vector field $2 \pi \frac{\partial}{\partial \xi_1}$, which coincides with the homotopy class of the flow of the vector field $2 \pi \frac{\partial}{\partial \xi_2}$.
Moreover, we have $\pi_k ( \Diff (S^3,\alpha) ) \cong \pi_k (S^3)$ for $k > 1$.
\end{prop}

\begin{cor}
If $M = S^3$, the helicity of a strictly contact vector field depends on the homotopy class rel end points of the isotopy it generates.
This holds true whether we consider homotopies in $\Diff (S^3,\alpha)$ or $\Diff (S^3,\alpha \wedge d\alpha)$.
\end{cor}

\begin{proof}
By Example \ref{exa:sphere}, we have $\Hel ( 4 \pi \frac{\partial}{\partial \xi_i} ) = 0 \not= 1 = \Hel (R_\alpha)$, for $i = 1, 2$, or $\Hel ( 2 \pi ( \frac{\partial}{\partial \xi_2} - \frac{\partial}{\partial \xi_1} ) ) = - 1 \not= 0 = \Hel (0)$.
Thus by Proposition~\ref{pro:homotopy-groups}, the helicity depends on the homotopy class rel end points.
The last statement follows from the induced homomorphism on fundamental groups.
\end{proof}

In other words, the helicity is not an invariant on the universal covering space of $\Diff_0 (M,\alpha)$, i.e.\ of the homotopy class rel end points of an isotopy.
We write $H \sim K$ if the isotopies $\Phi_H$ and $\Phi_K$ are homotopic rel end points through a homotopy of strictly contact isotopies.
That is, there exists a two-parameter family $\phi_{s,t}$ of strictly contact diffeomorphism, with $\phi_{s,0} = \id$ and $\phi_{s,1} = \phi \in \Diff (M,\alpha)$ for all $0 \le s \le 1$.
Denote by $X_{s,t}$ and $Y_{s,t}$ the vector fields defined by
	\[ \frac{d}{dt} \phi_{s,t} = X_{s,t} \circ \phi_{s,t}, \hspace{1cm} \frac{d}{ds} \phi_{s,t} = Y_{s,t} \circ \phi_{s,t}. \]
In particular, $Y_{s,0} = 0 = Y_{s,1}$.
Since the diffeomorphisms $\phi_{s,t}$ preserve $\alpha$, the vector fields $X_{s,t}$ and $Y_{s,t}$ are strictly contact.

\begin{lemma}{\cite{banyaga:gdp78}}
If $H \sim K$, then $c (H) = c (K)$.
\end{lemma}

\begin{proof}
It is well-known \cite{banyaga:sgd78} that
	\[ \frac{d}{ds} X_{s,t} = \frac{d}{dt} Y_{s,t} + [X_{s,t},Y_{s,t}]. \]
Contracting $\alpha$ with this equation and integrating over $[0,1]\times [0,1]\times M$ (the bracket of two functions has vanishing average value) proves the lemma.
\end{proof}

\begin{cor}
Suppose $H \sim K$ for two basic functions $H$ and $K$.
Then
	\[ \Hel (X_H) - \Hel (X_K) = 3 ( \| K \|_{L^2}^2 - \| H \|_{L^2}^2 ). \]
Thus the helicities of $X_H$ and $X_K$ are equal if and only if the $L^2$-norms of $H$ and $K$ coincide;
furthermore, if $\Hel (X_H) \ge \Hel (X_K)$, then $\| H \|_{L^2} \le \| K \|_{L^2}$, and if $\Hel(X_H) > \Hel(X_F)$, then $\| H \|_{L^2} < \| F \|_{L^2}$.
Conversely, suppose two basic functions $H$ and $K$ generate contact isotopies with the same end point.
If either $\Hel (X_H) \ge \Hel (X_K)$ and $\| H \|_{L^2} > \| K \|_{L^2}$, or $\Hel (X_H) > \Hel (X_K)$ and $\| H \|_{L^2} \ge \| K \|_{L^2}$, then $H \not\sim K$.
\end{cor}

For example, suppose $H \sim K$.
Then the helicity of the strictly contact vector field generating the composed isotopy $\Phi_H^{-1} \circ \Phi_K$ is $- 3 \| H - K \|_{L^2}^2 \le 0$.
If $X_H$ generates a loop and $\Hel (X_H) > 0$, the loop is not contractible.

\section{Suspensions of surface isotopies} \label{sec:suspensions}
In this section we improve, using different methods, results due to Gambaudo and Ghys~\cite{gambaudo:ea97}, relating the helicity to the Calabi invariant of surface isotopies.
Denote by $D^2 \subset \R^2$ the unit disk in the plane with polar coordinates $(r,\theta)$, where $0 \le r \le 1$, $0 \le \theta < 2 \pi$, and standard area (or symplectic) form $\omega = r \, dr \wedge d\theta$, and also consider the cylinder $D^2 \times \R$ with volume form $\omega \wedge dt$, where $t$ is the coordinate on the real line.
Let $\phi \in \Diff (D^2,\partial D^2,\omega)$ be an area preserving diffeomorphism that is the identity near the boundary $\partial D^2$ of the disk.
Consider the solid torus
	\[ T_\phi = D^2 \times \R \, / \, \{ (\phi^n (x), t) \sim (x, t + n) \mid n \in \Z \}, \]
with the induced volume form $\omega \wedge dt$, and denote by $p \colon D^2 \times \R \to T_\phi$ the canonical projection to the \emph{mapping torus} $T_\phi$.
The divergence-free vector field $\frac{\partial}{\partial t}$ projects to a divergence-free vector field $p_* (\frac{\partial}{\partial t}) = X (\phi)$, called the \emph{suspension} of the \emph{surface diffeomorphism} $\phi$ \cite{gambaudo:ea97}.
Clearly $\iota (X (\phi)) (\omega \wedge dt) = \omega = d\lambda$ for a one-form $\lambda$ on $D^2$, and the three-form $\lambda \wedge \omega$ vanishes for dimension reasons.
The situation becomes more interesting after embedding $T_\phi$ into standard $\R^3$.

Let $\{ \phi_t \}$ be an isotopy generated by a one-periodic smooth Hamiltonian function $F \colon D^2 \times \R \to \R$ that is compactly supported in the interior, with $\phi_0 = \id$ and time-one map $\phi_1 = \phi$.
This isotopy gives rise to a volume preserving embedding $T_\phi \to D^2 \times \R / \Z \hookrightarrow \R^3$ by composition of the map $S (\{ \phi_t \}) \colon T_\phi \to D^2 \times \R / \Z$ given by $(x,t) \mapsto (\phi_t (x),t)$ with an embedding $D^2 \times \R / \Z \hookrightarrow \R^3$ of the solid torus that is volume preserving with respect to the standard volume form $dV$ on $\R^3$.
An explicit (orientation preserving) embedding of $D^2 \times \R / \Z$ into $\R^3$ is given by
	\[ ((r,\theta),t) \mapsto ( (A + B r \cos \theta) \cos (2 \pi t), B r \sin \theta, (A + B r \cos \theta) \sin (2 \pi t) ), \]
which preserves total volume for an appropriate choice of constants $A > B > 0$.
By Moser's argument, it can be deformed to a volume preserving embedding $\tau$ (that preserves the boundary).
The vector field $X (\phi)$ defines a divergence-free vector field $S (\{ \phi_t \})_* (X (\phi)) = X (\{ \phi_t \}) = X + \frac{\partial}{\partial t}$ on the solid torus $D^2 \times \R / \Z$, where $X (x,t) = X_{F_t} (x)$ is the Hamiltonian vector field generating the isotopy $\{ \phi_t \}$.
We call $X (\{ \phi_t \})$ the \emph{suspension} of the \emph{surface isotopy} $\{ \phi_t \}$.
By identifying $D^2 \times \R / \Z$ with its image $\tau (D^2 \times \R / \Z)$ in $\R^3$, we can identity the vector fields $X (\{ \phi_t \})$ and $\tau_* X (\{ \phi_t \})$, and refer to the latter also as the suspension of the isotopy $\{ \phi_t \}$.

\begin{proof}[Proof of Theorem~\ref{thm:suspension}]
By restricting to differential forms on the torus $D^2 \times \R / \Z$ that are pull-backs of differential forms on the image of $\tau$ that extend to global differential forms on $\R^3$, the same argument as in the proof of Lemma~\ref{lem:change-vol-form} implies $\Hel (\tau_* X (\{ \phi_t \}),dV) = \Hel (X (\{ \phi_t \}),\omega \wedge dt)$.
On the other hand,
	\[ \iota_{X (\{ \phi_t \})} ( \omega \wedge dt ) = \iota_X \omega \wedge dt + \omega = dF_t \wedge dt + d \lambda = d (F_t \, dt + \lambda), \]
where $\lambda$ is a one-form on $D^2$ (that extends to $\R^2$) with $d\lambda = \omega$.
The primitive $F_t \, dt + \lambda$ extends to a global one-form on $\R^3$.
Moreover,
	\[ ( F_t \, dt + \lambda ) \wedge ( dF_t \wedge dt + \omega ) = F_t \, \omega \wedge dt - dF_t \wedge \lambda \wedge dt. \]
Since $F$ vanishes near the boundary, Stokes' theorem implies
	\[ 0 = \int_{D^2 \times \R / \Z} d (F_t \, \lambda \wedge dt) = \int_{D^2 \times \R / \Z} dF_t \wedge \lambda \wedge dt + F_t \, \omega \wedge dt, \]
and thus
	\[ \Hel (X (\{ \phi_t \}),\omega \wedge dt) = 2 \int_{D^2 \times \R / \Z} F_t \, \omega \wedge dt = 2 \int_0^1 \int_{D^2} F_t \, \omega \, dt = 2 \, \Cal (\phi). \qedhere\]
\end{proof}

\begin{exa}
Consider the solid torus $D^2 \times \R / \Z$ with coordinates $((r,\theta),t)$, where $0 \le r \le 1$, $0 \le \theta < 2 \pi$, and $0 \le t < 1$ (considering a disk of arbitrary radius corresponds to rescaling the area form).
Let $\rho \colon [0,1] \to \R$ be a smooth function that is identically zero near $r = 1$, and consider the area preserving diffeomorphism $\phi_\rho \colon D^2 \to D^2$ defined by $(r,\theta) \mapsto (r,\theta + \rho (r))$ for $r > 0$, and $\phi_\rho (0) = 0$, where $0$ denotes the origin in $\R^2$.
The suspension of the isotopy $\{ \phi_\rho^t \} = \{ \phi_{t \rho} \}$ is the vector field
	\[ X (\{ \phi_{t \rho} \}) = \rho (r) \frac{\partial}{\partial \theta} + \frac{\partial}{\partial t}, \]
cf.\ \cite[Section 1.4]{gambaudo:sac01}, and its generating Hamiltonian function is
	\[F(r,\theta) = \int_r^1 s \rho(s) \, ds, \]
\cite[Example 4.2]{mueller:ghh07}.
An easy computation shows
	\[ \iota_{X (\{ \phi_{t \rho} \})} (r \, dr \wedge d\theta \wedge dt) = d \left( \frac{1}{2} r^2 \, d\theta + \left( \int_r^1 s \rho (s) \, ds \right) dt \right) = d\beta_\rho. \]
Then
	\[ \beta_\rho \wedge d\beta_\rho = \left( \frac{1}{2} r^2 \rho (r) + \int_r^1 s \rho (s) \, ds \right) r \, dr \wedge d\theta \wedge dt. \]
Using integration by parts for the second summand, we find
	\[ \Hel (X (\{ \phi_{t \rho} \})) = \int_0^1 \int_0^{2 \pi} \int_0^1 \left( \frac{1}{2} r^3 \rho (r) + r \int_r^1 s \rho (s) \, ds \right) dr d\theta dt = 2 \pi \int_0^1 r^3 \rho (r) \, dr. \]
\end{exa}

Let $\phi_1$, $\phi_2 \in \Diff (D^2,\partial D^2,\omega)$ be two area preserving diffeomorphisms that are the identity near the boundary, and consider the corresponding solid tori $T_{\phi_1}$ and $T_{\phi_2}$.
Let $\{ \phi_1^t \}$ and $\{ \phi_2^t \}$ be Hamiltonian isotopies as above with $\phi_1^0 = \phi_2^0 = \id$, and time-one maps $\phi_1^1 = \phi_1$ and $\phi_2^1 = \phi_2$.
Again consider $X (\{ \phi_1^t \}) = X_1 + \frac{\partial}{\partial t}$ and $X (\{ \phi_2^t \}) = X_2 + \frac{\partial}{\partial t}$ as vector fields on $D^2 \times \R / \Z$, and embed the solid torus into the three-sphere by
	\[ \tau^1 \colon ( (r,\theta),t ) \mapsto \left( \frac{1}{2} \sin^{-1} r, \theta + 2 \pi t, 2 \pi t \right) =  (\eta, \xi_1,\xi_2) \in S^3 \]
(rather than the standard $( (r,\theta),t ) \mapsto (\frac{1}{2} \sin^{-1} r, \theta, 2 \pi t)$) and
	\[ \tau^2 \colon ( (r,\theta),t ) \mapsto \left( \frac{1}{2} (\pi - \sin^{-1} r), 2 \pi t, \theta + 2 \pi t \right) =  (\eta, \xi_1,\xi_2) \in S^3, \]
so that the vector fields $\tau_*^1 ( X (\{ \phi_1^t \}) )$ and $\tau_*^2 ( X (\{ \phi_2^t \}) )$ coincide along their common boundary $\{ \eta = \pi/4 \}$.
Here the image of $\tau^1$ is the solid torus $\{ \eta \le \pi/4 \}$ in $S^3$, and the image of $\tau^2$ is the solid torus $\{\eta \ge \pi/4 \}$.
Denote their sum by $X (\{ \phi_1^t \},\{ \phi_2^t \})$.
This divergence-free vector field on the three-sphere $S^3$ is called the \emph{double suspension} of $\{ \phi_1^t \}$ and $\{ \phi_2^t \}$.
Up to scaling, the volume form $\mu$ on $S^3$ obtained from gluing together the two copies of the solid torus is the standard one, and has total volume equal to $2 \pi$.

\begin{proof}[Proof of Theorem~\ref{thm:double-suspension}]
By Lemma~\ref{lem:change-vol-form},
	\[ \Hel (X (\{ \phi_1^t \},\{ \phi_2^t \});dV) = \Hel (X (\{ \phi_1^t \},\{ \phi_2^t \});\pi \mu) = \pi^2 \cdot \Hel (X (\{ \phi_1^t \},\{ \phi_2^t \});\mu). \]
By construction, $X (\{ \phi_1^t \},\{ \phi_2^t \}) = \tau_*^1 (X_1) + \tau_*^2 (X_2) + R_\alpha$, so that
	\[ \Hel (X (\{ \phi_1^t \},\{ \phi_2^t \})) = \Hel (\tau_*^1 (X_1) + \tau_*^2 (X_2)) + 2 \Rel (\tau_*^1 (X_1) + \tau_*^2 (X_2),R_\alpha) + \Hel (R_\alpha). \]
For the first term, we compute
\begin{align*}
	& \Hel (\tau_*^1 (X_1) + \tau_*^2 (X_2);\mu) \\
	& = \Hel (\tau_*^1 (X_1);\mu) + 2 \Rel (\tau_*^1 (X_1),\tau_*^2 (X_2);\mu) + \Hel (\tau_*^2 (X_2);\mu) \\
	& = \Hel (X_1,\omega \wedge dt) + \Hel (X_2,\omega \wedge dt) = 0,
\end{align*}
since $\tau_*^1 (X_1)$ and $\tau_*^2 (X_2)$ have disjoint supports on $S^3$, and by the last statement of Theorem~\ref{thm:suspension}.
Moreover,
\begin{align*}
	\Rel (\tau_*^1 (X_1) + \tau_*^2 (X_2),R_\alpha;\mu) &= \Rel (\tau_*^1 (X_1),R_\alpha;\mu) + \Rel (\tau_*^2 (X_2),R_\alpha;\mu) \\
	&= \Rel (X_1,\frac{\partial}{\partial t};\omega \wedge dt) + \Rel (X_2,\frac{\partial}{\partial t};\omega \wedge dt) \\
	&= \Cal (\phi_1) + \Cal (\phi_2).
\end{align*}
Since $\Hel (R_\alpha) = ( \vol (S^3) )^2$, combining all of the above proves the claim.
\end{proof}

\section{Continuous contact isotopies} \label{sec:topo-contact}
Let $(M,\xi)$ be a closed contact manifold equipped with a contact form $\alpha$.
A continuous isotopy $\Phi = \{ \phi_t \}$ in the group Homeo$(M)$ of homeomorphisms is a \emph{continuous strictly contact isotopy} if there exists a Cauchy sequence of smooth basic contact Hamiltonian functions $H_i \colon [0,1] \times M \to \R$, such that the sequence $\Phi_{H_i}$ of strictly contact isotopies converges uniformly to $\Phi$.
Here the norm \cite{banyaga:ugh11} used to define the metric on the space of contact Hamiltonian functions is
\begin{equation} \label{eqn:hofer-norm}
	\| H \| = \max_{0 \le t \le 1} \left( \max_{x \in M} H (t,x) - \min_{x \in M} H (t,x) + | c (H_t) | \right),
\end{equation}
which means the Cauchy sequence $H_i$ converges uniformly.
It is also possible to replace the maximum over $0 \le t \le 1$ by the integral over the interval $[0, 1]$, but we will restrict to the former case in this article.
For a detailed study of continuous contact isotopies and related notions, see \cite{banyaga:ugh11, mueller:gcd11}.

Suppose now $(M,\alpha)$ is regular.
Then the \emph{continuous contact Hamiltonian function} $H = \lim_i H_i$ associated to the continuous strictly contact isotopy $\Phi$ is unique \cite{banyaga:ugh11}.
In other words, if $H_i$ and $K_i$ are two Cauchy sequences of smooth basic contact Hamiltonian functions with $\lim_i \Phi_{H_i} = \Phi = \lim_i \Phi_{K_i}$, then we must have $\lim_i H_i = \lim_i K_i$.
Thus the limit $\Hel (\Phi)$ of the sequence $\Hel (X_{H_i})$ exists and does not depend on the sequence $H_i$ but only on the continuous strictly contact isotopy $\Phi$.
That proves Corollary~\ref{cor:helicity}, and that Definition~\ref{def:helicity} is well-defined.

We point out that the contact Hamiltonian functions $H_i$ are time-dependent in general, even if the limit $H$ is autonomous.
However, by our earlier remark the helicity of a time-dependent divergence-free vector field that is exact at all times $t$ is well-defined, and Definition~\ref{def:helicity} makes sense for any continuous strictly contact isotopy.

Conversely, a continuous strictly contact isotopy is uniquely determined by its continuous contact Hamiltonian function \cite{banyaga:ugh11}.
To be more precise, suppose two sequences of smooth basic contact Hamiltonian functions $H_i$ and $K_i$ generate two sequences $\Phi_{H_i}$ and $\Phi_{K_i}$ of uniformly convergent strictly contact isotopies.
If $\lim_i H_i = \lim_i K_i$, then $\lim_i \Phi_{H_i} = \lim_i \Phi_{K_i}$.
Denote the common limit by $\Phi$.
In the terminology of topological (strictly) contact dynamics, the continuous Hamiltonian function $H = \lim_i H_i$ `generates' the isotopy $\Phi$, and we write $\Phi_H = \Phi$.

By the above uniqueness theorems, the continuous strictly contact isotopy $\Phi_H$ is a one-parameter subgroup if and only if its continuous contact Hamiltonian function is autonomous.
Furthermore, a continuous contact Hamiltonian function $H$ is invariant under the Reeb flow, and we call $H$ a \emph{continuous basic function} on $M$ \cite{mueller:gcd11}.
In particular, there exists a unique function $F$ on $B$ such that $H = p^* F = F \circ p$.
This function $F$ is a continuous Hamiltonian function in the sense explained in the next section.
Moreover, the $S^1$-extension (\ref{eqn:s^1-ext}) extends to so called strictly contact homeomorphisms of $M$, i.e.\ time-one maps of continuous strictly contact isotopies, and Hamiltonian homeomorphisms of $B$ \cite{banyaga:ugh11}.
The latter were defined and studied in \cite{mueller:ghh07, mueller:ghl08, mueller:ghc08}.
See the next section for a brief summary.
Example~\ref{exa:lift} concerning horizontal lifts can be generalized verbatim to continuous Hamiltonian isotopies.

Suppose $\phi \in \Homeo (M)$ is the uniform limit $\phi = \lim_i \phi_i$ of a sequence of strictly contact diffeomorphisms, i.e.\ $\phi_i^* \alpha = \alpha$ for all $i$.
We denote the group of strictly contact diffeomorphisms by $\Diff (M,\alpha)$, and the group of limit homeomorphisms by $\overline{\Diff} (M,\alpha)$.
By rigidity, if a homeomorphism $\phi \in \overline{\Diff} (M,\alpha)$ is smooth, then $\phi \in \Diff (M,\alpha)$ \cite{mueller:gcd11}, justifying our notation.

\begin{proof}[Proof of Theorem~\ref{thm:strictly-contact-homeo-conjugation}]
Let $\Phi$ be a continuous strictly contact isotopy with continuous contact Hamiltonian function $H$, and $H_i$ be a Cauchy sequence with limit $H$ and $\Phi_{H_i} \to \Phi$ uniformly, whose existence is guaranteed by the definition of a continuous strictly contact isotopy.
Then the conjugated smooth isotopy $\phi_i^{-1} \circ \Phi_{H_i} \circ \phi_i$ has the smooth contact Hamiltonian function $H_i \circ \phi_i$, and moreover, it converges to $\phi^{-1} \circ \Phi \circ \phi$ uniformly, and $H_i \circ \phi_i$ converges to $H \circ \phi$ in the metric defined by (\ref{eqn:hofer-norm}) \cite{banyaga:ugh11}.
This extension of the usual transformation law provides further justification for our notation.
Since $\phi$ preserves the (measure induced by the) volume form $\alpha \wedge d\alpha$ on $M$, the change of variables formula shows that $\phi$ preserves the average value $c$ of a function on $M$.
Thus the following identities hold.
\begin{align*}
	\Hel (\phi \circ \Phi \circ \phi^{-1}) &= \left( 4 c^2 (H \circ \phi) - 3 c ( (H \circ \phi)^2 ) \right) \cdot \vol (M) \\
	&= \left( 4 c^2 (H) - 3 c (H^2) \right) \cdot \vol (M)\\
	&= \Hel (\Phi) \qedhere
\end{align*}
\end{proof}

\section{Continuous Hamiltonian isotopies} \label{sec:topo-ham}
We briefly recall the definition of a continuous Hamiltonian isotopy, which is similar to the case of a continuous contact isotopy discussed in the previous section.
See \cite{mueller:ghh07, mueller:ghl08, mueller:ghc08} for details.
Let $(B^{2n},\omega)$ be a closed and connected symplectic manifold.
Recall from Section~\ref{sec:contact} the one-one correspondence between Hamiltonian vector fields $X = X_F$, and smooth mean-value zero normalized functions $F$ on $M$, given by the relation $\iota_X \omega = dF$.
In this context, a smooth function on $M$ is generally referred to as a \emph{Hamiltonian} function.
Suppose $F_i \colon [0,1] \times B \to \R$ is a sequence of smooth mean value zero normalized time-dependent Hamiltonian functions, and denote by $\Phi_{F_i} = \{ \phi_{F_i}^t \}$ the sequence of smooth Hamiltonian isotopies corresponding to the Hamiltonian functions $F_i$, i.e.\ the isotopies generated by the vector fields $X_{F_i}$.
If $F_i$ is a Cauchy sequence with respect to the metric induced by the norm (\ref{eqn:hofer-norm}), and $\Phi_{F_i}$ converges uniformly to a continuous isotopy $\Phi = \{ \phi_t \}$ of homeomorphisms, then $\Phi$ is called a \emph{continuous Hamiltonian isotopy}, and the limit $F = \lim F_i$ is called a {\em continuous Hamiltonian function}.
The group of time-one maps of continuous Hamiltonian isotopies is denoted $\Hameo(B,\omega)$.
Note that we assume the Hamiltonian functions have mean value zero (with respect to the canonical volume form $\omega^n$), so that the term $c (F (t,\cdot))$ in (\ref{eqn:hofer-norm}) vanishes in the present situation.
It is again possible to work with the (weaker) norm obtained by replacing the maximum by the time average over $0 \le t \le 1$.
These definitions make sense for noncompact manifolds and manifolds with nonempty boundary, provided one considers only Hamiltonian functions that are compactly supported in the interior of $B$, and adjusts the definition accordingly.
In this case, the mean value of a Hamiltonian function need no longer vanish identically, it is instead normalized by the requirement of having compact support in the interior.

A continuous Hamiltonian isotopy is uniquely determined by its generating Hamiltonian function \cite{mueller:ghh07}.
That is, suppose that a sequence $\Phi_{F_i}$ of smooth Hamiltonian isotopies, generated by normalized smooth Hamiltonian functions $F_i$, converges uniformly to an isotopy $\Phi$ of homeomorphisms, and that the Hamiltonian functions converge to a continuous function $F$ uniformly.
If another such sequence $\Phi_{G_i}$ satisfies $\lim_i F_i = \lim_i G_i$, then $\lim_i \Phi_{F_i} = \lim_i \Phi_{G_i}$, and we denote $\Phi_F = \Phi$.

Conversely, L.~Buhovsky and S.~Seyfaddini \cite{buhovsky:ugh11} generalized (and simplified the proof of) a previous result by C.~Viterbo \cite{viterbo:ugh06} on the uniqueness of the `generating Hamiltonian' $F$.
That is, if $F_i \to F$ and $G_i \to G$, and the isotopies $\Phi_{F_i}$ and $\Phi_{G_i}$ have the same uniform limit $\Phi$, then $F = G$.
When $B = D^2$, the real number
\begin{equation} \label{eqn:calabi-topological}
	\int_0^1 \int_{D^2} F_t \, \omega \, dt
\end{equation}
is well-defined, and depends only on the continuous Hamiltonian isotopy $\Phi$.
If $\Phi$ is a smooth Hamiltonian isotopy, it equals the Calabi invariant of the time-one map of $\Phi$.
We thus call (\ref{eqn:calabi-topological}) the \emph{Calabi invariant} $\Cal (\Phi)$ of the isotopy $\Phi$, and note $\Cal (\Phi) = \lim_i \Cal (\Phi_{F_i})$ for a (and thus any) sequence of Hamiltonian isotopies converging to $\Phi$ in the sense of the definition of a continuous Hamiltonian isotopy.
Note that every smooth area-preserving isotopy of $(D^2,\partial D^2,\omega)$ is Hamiltonian.

Suppose $\Phi = \{ \phi_t \}$ is a continuous Hamiltonian isotopy, $F_i$ a Cauchy sequence in the sense explained above, and the smooth Hamiltonian isotopies $\Phi_{F_i}$ converge uniformly to $\Phi$.
Consider the suspensions $X (\Phi_{F_i})$ defined in Section~\ref{sec:suspensions}.
These do not necessarily converge to a (continuous) vector field on $D^2 \times \R / \Z$ as $i \to \infty$.
However, the flows (with time-$s$ maps) $(x,t) \mapsto (\phi_{F_i}^{t + s} \circ (\phi_{F_i}^t)^{-1} (x), t + s)$ converge uniformly to the flow $(x,t) \mapsto (\phi_{t + s} \circ \phi_t^{-1} (x),t + s)$, which we call the \emph{suspension} of $\Phi$.
As remarked above, $\Cal (\Phi_{F_i}) \to \Cal (\Phi)$, independently of the choice of sequence $F_i$ in Definition~\ref{def:helicity-twice-calabi}.
That proves Corollary~\ref{cor:helicity-twice-calabi}, shows that Definition~\ref{def:helicity-twice-calabi} is well-defined, and extends the definition in the case of a smooth Hamiltonian isotopy.
However, recall again that in general the helicity is not $C^0$-continuous with respect to the isotopy $\Phi$; if a sequence of isotopies $\Phi_i$ converges only uniformly to $\Phi$, their helicities need not converge.

The group $\Sympeo (M,\omega)$ of symplectic homeomorphisms is by definition the $C^0$-closure of the group $\Symp (M,\omega) = \{ \phi \in \Diff(M) \mid \phi^*\omega = \omega\}$ of symplectic diffeomorphisms in $\Homeo (M)$ \cite{mueller:ghh07}.
The usual transformation law continues to hold for homeomorphisms, i.e.\ $\phi^{-1} \circ \Phi \circ \phi = \{ \phi^{-1} \circ \phi_t \circ \phi \}$ has continuous Hamiltonian function $F \circ \phi$ for any continuous Hamiltonian isotopy $\Phi = \{ \phi_t \}$ with continuous Hamiltonian function $F$, and symplectic homeomorphism $\phi$.
Again by rigidity, an element of $\Sympeo (M,\omega)$ that is smooth belongs to $\Symp (M,\omega)$, which together with the transformation law justifies our notation.

The Calabi invariant of $\Phi$ is conjugation-invariant by area preserving diffeomorphisms of the two-disk.
Any area preserving homeomorphism can be approximated uniformly by diffeomorphisms { \cite{munkres:osp59, munkres:osp60, munkres:hos65, hirsch:ots63}}, and thus by area preserving (or symplectic) diffeomorphisms \cite{oh:cmp06, sikorav:avp07}.
Therefore the Calabi invariant (\ref{eqn:calabi-topological}) of a (smooth or continuous) Hamiltonian isotopy is invariant under conjugation by any area preserving homeomorphism of the two-disk.
For smooth isotopies, this was known to Gambaudo and Ghys, see \cite{gambaudo:ea97} for a different proof.

\begin{proof}[Proof of Theorem~\ref{theo:suspension-conjugation}]
Denote the suspensions by ${\tilde \phi}^s (x,t) = (\phi_{t + s} \circ \phi_t^{-1} (x),t + s)$ and ${\tilde \psi}^s (x,t) = (\psi_{t + s} \circ \psi_t^{-1} (x),t + s)$, and write ${\tilde \varphi} (x,t) = (\varphi (x),t)$.
By hypothesis,
	\[ {\tilde \psi}^s (x,t) = {\tilde \varphi} \circ {\tilde \phi}^s \circ {\tilde \varphi}^{-1} (x,t) = (\varphi \circ (\phi_{t + s} \circ \phi_t^{-1}) \circ \varphi^{-1} (x),t + s). \]
The right-hand side is the suspension of the isotopy $\varphi \circ \Phi \circ \varphi^{-1} = \{ \varphi \circ \phi_t \circ \varphi^{-1} \}$, and thus has helicity $2 \Cal (\varphi \circ \Phi \circ \varphi^{-1}) = 2 \Cal (\Phi)$ by conjugation invariance, which in turn equals the helicity of the suspension of $\Phi$.
\end{proof}

\section{Topologically conjugate diffeomorphisms}\label{sec:conjugate-diffeomorphisms}
Recall Arnold's first question presented in the introduction.
Suppose (the volume preserving isotopies generated by) two smooth exact divergence-free vector fields $X$ and $Y$ are topologically conjugate.
That means there exists a (volume preserving) homeomorphism $\psi$ such that $\{ \psi \circ \phi_X^t \circ \psi^{-1} \} = \{ \phi_Y^t \}$.
If $\psi$ is a $C^1$-diffeomorphism, this is equivalent to $\psi_* X = Y$, and it is easy to see that the helicities of $X$ and $Y$ coincide (Lemma~\ref{lem:change-vol-form}).
Does this identity hold in general, even if $\psi$ is not a $C^1$-diffeomorphism, and thus $\psi_* X$ is not well-defined?
Theorem~\ref{thm:strictly-contact-homeo-conjugation} and Theorem~\ref{theo:suspension-conjugation} provide positive answers in two particular cases coming from the contact geometry of regular contact three-manifolds, and the symplectic geometry of surfaces.
The discussion in this section and the next is intended to illustrate these results.

The proof of the following algebraic lemma is trivial.
For $G$ a group, we denote by $Z_g = \{ c \in G \mid g c = c g \}$ the centralizer of $g \in G$.

\begin{lemma}
If for $a, b, c$ elements of some group $G$ we have $c a c^{-1} = b$, then $d a d^{-1} = b$ if and only if $d \in c \cdot Z_a = Z_b \cdot c$.
\end{lemma}

The lemma applied to the group $G = \Homeo (M)$ says if $\psi \circ \phi \circ \psi^{-1} = \varphi$, then in general $\psi$ is not unique.
Note that it is trivial to produce examples of two diffeomorphisms or isotopies of diffeomorphisms that are topologically conjugate.
For example, if a diffeomorphism $\phi$ equals the identity on some open set $U$, and $\psi$ is a diffeomorphism outside an open set $V \subset U$ but non-smooth inside $V$, then $\psi \circ \phi \circ \psi^{-1}$ is a diffeomorphism.
Similar examples can be constructed if $\phi$ is the identity on some factor of a product manifold.
The actual problem is to find examples of topologically conjugate diffeomorphisms or isotopies of diffeomorphisms that are not conjugated by a diffeomorphism.
Indeed it appears to be quite rare a situation that two diffeomorphisms are conjugated by a homeomorphism but not a $C^1$-diffeomorphism.\footnote{We thank \'E.~Ghys for making this observation during a private conversation at Edifest, ETH Z\"urich in November 2010, and for making us aware of Furstenberg's example.}

The following construction is due to Furstenberg~\cite{furstenberg:set61, rouhani:ftt90}.
Let $\theta$ be an irrational number, $d$ an integer, and $f$ a smooth function on $S^1$.
A {\em Furstenberg transformation} $\phi_{\theta,d,f}$ is a diffeomorphism of $T^2$ defined by
	\[ \phi_{\theta,d,f} (x, y) = (x e^{2 \pi i \theta}, x^d y e^{2 \pi i f (x)}). \]
Furstenberg transformations are always area preserving (with respect to the standard area form $dx \wedge dy$), and {\em minimal} (that is, every orbit is dense in the torus $T^2$) provided $d$ is nonzero \cite{furstenberg:set61, rouhani:ftt90}.
The following lemma is essentially contained in \cite{kodaka:aft95}.

\begin{lemma} \label{lem:kodaka}
Let $\theta$ be an irrational number, $d$ a nonzero integer, and $f$ a smooth function on $S^1$.
Consider the Furstenberg transformations $\phi_{\theta,d,f}$ and $\phi_{\theta,d,0}$ of $T^2$.
There exists a continuous map $\psi \colon T^2 \to T^2$ that satisfies the identity $\psi \circ \phi_{\theta,d,f} = \phi_{\theta,d,0} \circ \psi$ if and only if $f$ can be split with respect to the circle action $x \mapsto e^{2 \pi i \theta} x$ on $S^1$, i.e.\ there exists a continuous function $g$ on $S^1$ that satisfies the equation $g (x) - g (e^{2 \pi i \theta} x) = f (x) - \eta$, where $\eta$ denotes the average value of $f$ (with respect to the measure induced by $dx$).
In that case,
	\[ \psi (x,y) = \left( x e^{2 \pi i \frac{m \theta + \eta + k}{d}}, x^m y e^{2 \pi i g (x)} \right), \]
where $m$ and $k$ are integers.
In particular, $\psi$ is (a posteriori) an area preserving homeomorphism, and $\phi_{\theta,d,f}$ is topologically conjugate to $\phi_{\theta,d,0}$.
Moreover, the function $g$ is unique up to (adding) a real constant.
Thus if $f$ is smooth and $g$ is not $C^1$, then $\phi_{\theta,d,f}$ is topologically conjugate but not $C^1$-conjugate to $\phi_{\theta,d,0}$.
\end{lemma}

\begin{proof}
The `if' part is a straightforward computation.
For the `only if' part, suppose $g_1$ and $g_2$ are continuous functions with $g_i (x) - g_i (e^{2 \pi i \theta} x) = f (x) - \eta$ for $i = 1, 2$.
Their difference then satisfies $(g_1 - g_2) (x) = (g_1 - g_2) (e^{2 \pi i \theta} x)$ with $\theta$ irrational, and by continuity, $g_1 - g_2$ is constant.
By the homotopy lifting theorem we may write
	\[ \psi (x,y) = (x^{m_1} y^{n_1} e^{2 \pi i F_1 (x,y)}, x^{m_2} y^{n_2} e^{2 \pi i F_2 (x,y)}), \]
for integers $m_1, n_1, m_2, n_2$, and smooth functions $F_1, F_2$ on the torus.
Calculating explicitly $\psi \circ \phi_{\theta,d,f} = \phi_{\theta,d,0} \circ \psi$, and using that $\theta$ is irrational and $\phi_f$ is area preserving and minimal yields
	\[ \psi (x,y) = \left( x e^{2 \pi i \frac{m \theta + \eta + k}{d}}, x^m y e^{2 \pi i F (x,y)} \right), \]
where $F (x,y) - F (\phi_f (x,y)) = f (x) - \eta$ and $F \colon T^2 \to \R$ is a continuous function \cite{kodaka:aft95}.\footnote{Kodaka gave the proof in the case $d= 1$, but the general case is proved verbatim.
He moreover assumed that a priori $\psi$ is a homeomorphism; however this is not necessary.}
The function $G$ defined by $G (x,y) = g (x)$ also satisfies the equality $G (x,y) - G (\phi_f (x,y)) = f (x) - \eta$.
Subtracting shows the continuous function $F - G$ is constant since $\phi_{\theta,d,f}$ is minimal.
Then $F (x,y) = g (x) + c$ as claimed.
This shows the map $\psi$ has the required form, and it is easy to see that it is injective and surjective, and thus a homeomorphism (since $T^2$ is compact and Hausdorff).
\end{proof}

\begin{exa} {\cite{furstenberg:set61}}
For an irrational number $\theta$, choose a sequence of integers $n_k \ge 2^k$,  $k \ge 1$, such that $0 < n_k \theta - [ n_k \theta ] \le 2^{- n_k}$, where $[x]$ as usual denotes the greatest integer less than or equal to $x$.\footnote{Furstenberg originally constructed only one such number $\theta$, but by Poincar\'e's recurrence theorem, this can be done for any irrational number $\theta$.}
Define $n_k = - n_{- k}$ for $k < 0$.
Then the real function $f \colon S^1 \to \R$ defined by
	\[ f (e^{2 \pi i t}) = \sum_{k \not= 0} \frac{1}{k^2} \left( 1 - e^{2 \pi i n_k \theta} \right) \, e^{2 \pi i n_k t}, \]
is a smooth function on $S^1$ with mean value zero.
Now define the real function $g \colon S^1 \to \R$ by
	\[ g (e^{2 \pi i t}) = \sum_{k \not= 0} \frac{1}{k^2} e^{2 \pi i n_k t}, \]
which is continuous in $x \in S^1$, but not $C^1$-smooth.
It is immediate to check that $g (x) - g (e^{2 \pi i \theta} x) = f (x)$.
\end{exa}

A necessary and sufficient criterion for when a function can be split (with respect to some minimal homeomorphism) was proved in \cite[page 135]{gottschalk:td55}, see also \cite{rouhani:ftt90, kodaka:aft95}.
Examples of such functions are most easily constructed as above using Fourier series \cite[Lemma 2.1]{kodaka:tsc95} and Plancherel's theorem, where the number $\eta$ is the coefficient of the constant term.

The following proposition generalizes the preceding example (in which $M$ is void) to trivial $T^2$-bundles of any dimension.

\begin{prop} \label{pro:trivial-bundles}
For $M$ a smooth manifold, there exist pairs of diffeomorphisms of $M \times T^2$ that are conjugated by a homeomorphism but not by any $C^1$-diffeomorphism.
If moreover $M$ supports a volume form $\mu$, and $M \times T^2$ is equipped with the product volume form $\mu \wedge dx \wedge dy$, then there are pairs of diffeomorphisms as above which in addition are volume preserving, and the conjugating homeomorphism may also be chosen to preserve volume.
\end{prop}

\begin{proof}
A straightforward computation shows that
	\[ (\id \times \psi) \circ (\id \times \phi_{\theta,d,f}) = (\id \times \phi_{\theta,d,0}) \circ (\id \times \psi), \]
where $\phi_{\theta,d,f}$, $\phi_{\theta,d,0}$, and $\psi$ are as in the example above.
We may in fact replace the second and third identity map by any homeomorphism $\phi$ of $M$ here and in the argument below, and also choose $g (p,x) = g (x) + c (p)$ for a continuous function $c$ on $M$.
That proves the existence part of the proposition.

Let $\psi = \psi_1 \times \psi_2 \colon M \times T^2 \to M \times T^2$ be a $C^1$-diffeomorphism, where the factors $\psi_1 \colon M \times T^2 \to M$ and $\psi_2 \colon M \times T^2 \to T^2$ are both $C^1$-smooth maps, and assume that $\psi \circ (\id \times \phi_{\theta,d,f}) = (\id \times \phi_{\theta,d,0}) \circ \psi$.
For fixed $p \in M$, consider the restriction $\psi_2^p = \psi_2 (p,\cdot)$ of $\psi_2$ to the fiber over $p$.
By a routine computation, we have $\psi_2^p \circ \phi_{\theta,d,f} = \phi_{\theta,d,0} \circ \psi_2^p$.
But by Lemma~\ref{lem:kodaka}, $\psi_2^p$ cannot be $C^1$-smooth, a contradiction.
\end{proof}

Alternatively, the identity $\psi \circ (\id \times \phi_{\theta,d,f}) = (\id \times \phi_{\theta,d,0}) \circ \psi$ also implies $\psi_1^p = \psi_1^p \circ \phi_{\theta,d,f}$, where $\psi_1^p \colon T^2 \to M$ is the restriction of $\psi_1$ to the fiber over $p$.
Since $\phi_{\theta,d,f}$ is minimal ($d \not= 0$), this implies $\psi_1$ only depends on $p$ but not on $(x, y) \in T^2$.
Thus for fixed $p \in M$, the map $\psi_2^p$ is a local diffeomorphism.
Its image is open and closed, and therefore $\psi_2^p$ is surjective.
Since it is also injective, it is a diffeomorphism of $T^2$, and we may proceed as above to derive a contradiction.
However, the proof given above shows there does not even exist a $C^1$-smooth \emph{map} $\psi$ such that $\psi \circ (\id \times \phi_{\theta,d,f}) = (\id \times \phi_{\theta,d,0}) \circ \psi$.
The following result can be proved similarly to Lemma~\ref{lem:kodaka}.

\begin{lemma}
If $M$ is connected and simply-connected, or $M = T^n$ in Proposition~\ref{pro:trivial-bundles}, then the homeomorphism $\psi$ conjugating $\id \times \phi_{\theta,d,f}$ and $\id \times \phi_{\theta,d,0}$ is of the form
	\[ \psi (p,(x,y)) = \left( \psi_1 (p), (x e^{2 \pi i \frac{m \theta + \eta + k}{d}}, x^m y e^{2 \pi i (g (x) + c (p))}) \right) \]
respectively
	\[ \psi (p,(x,y)) = \left( \psi_1 (p), (x e^{2 \pi i \frac{m \theta + \eta + k}{d}}, x^m y e^{2 \pi i (g (x) + c (p))} p_1^{q_1} \cdots p_n^{q_n}) \right), \]
for a homeomorphism $\psi_1$ of $M$, integers $m$, $k$, and $q_1, \ldots, q_n$, and a continuous function $c$ on $M$.
$\psi$ is volume preserving (with respect to a product volume form on $M \times T^2$) if and only if $\psi_1$ is.
\end{lemma}

Conversely, we have the following example.

\begin{exa}
Suppose $f = \eta$ is constant, and $\psi \circ \phi_{\theta,d,\eta} = \phi_{\theta,d,0} \circ \psi$.
Then again $\psi$ has the above form for a function $g$ that satisfies $g (x) - g (e^{2 \pi i \theta} x) = f (x) - \eta = 0$.
Assuming $\psi$ is continuous, $g$ must be continuous, and thus constant.
Therefore
	\[ \psi (x,y) = (x e^{2 \pi i \frac{m \theta + \eta + k}{d}}, x^m y e^{2 \pi i c}) \]
for some $c \in \R$, and is in particular an area preserving diffeomorphism.
That provides an example of two area preserving diffeomorphisms that are conjugated only by (area preserving) diffeomorphisms.
\end{exa}

A thorough study of topological and smooth conjugacy of diffeomorphisms of the circle $S^1$ is carried out in the book by A.~Katok and B.~Hasselblatt \cite{katok:imt96}.
In particular, for any integer $r \ge 0$, the authors construct examples of $C^{r+1}$-smooth diffeomorphisms of $S^1$ that are only conjugated by diffeomorphisms (or homeomorphisms, when $r=0$) of class $C^r$.

\section{Topologically conjugate smooth dynamical systems} \label{sec:conjugate-systems}
We would like to point out that none of the diffeomorphisms in Section~\ref{sec:conjugate-diffeomorphisms} that are conjugated by homeomorphisms (and only homeomorphisms) are isotopic to the identity, since the integer $d$ in the definition of the Furstenberg transformation is nonzero.
We now construct examples of smooth Hamiltonian and strictly contact {\em isotopies} that are conjugated by a homeomorphism but not by symplectic or contact $C^1$-diffeomorphisms, respectively.
We begin by recalling some facts previously used in this work in the form of a well-known and easy to verify lemma.
We state it for autonomous vector fields, but the conclusions of the lemma are equally valid for time-dependent vector fields.

\begin{lemma}
Let $X$ and $Y$ be smooth vector fields on $M$, and $\phi$ a diffeomorphism of $M$.
Then $\phi \circ \phi_X^t \circ \phi^{-1} = \phi_Y^t$ for all $t$ if and only if $\phi_* X = Y$.
If $X =  X_H$ and $Y = X_F$ are Hamiltonian (with respect to some symplectic form if $M$ has even dimension) or strictly contact (with respect to some contact form if $M$ is odd-dimensional), and $\phi$ is symplectic or contact, respectively, then $\phi \circ \phi_H^t \circ \phi^{-1} = \phi_F^t$ for all $t$ if and only if $e^h H = \phi^* F = F \circ \phi$.
Here $\phi^* \alpha = e^h \alpha$ if $\phi$ is a contact diffeomorphism, and $h = 0$ if $\phi$ is symplectic.
\end{lemma}

The proposition we are stating next by-passes the reference to the vector fields, and directly relates the isotopies to their Hamiltonian functions.

\begin{prop}{\cite{mueller:ghh07, banyaga:ugh11}}\label{pro:uniqueness}
Suppose $\Phi_H$ and $\Phi_F$ are continuous or smooth Hamiltonian isotopies (of a symplectic manifold) or strictly contact isotopies (of a regular contact manifold), and $\phi$ is a symplectic homeomorphism or the uniform limit of strictly contact diffeomorphisms.
Then $\phi \circ \phi_H^t \circ \phi^{-1} = \phi_F^t$ for all $t$ if and only if $H = F \circ \phi$.
\end{prop}

\begin{exa}\label{exa:non-conjugate-surface}
Let $(M^2,\omega)$ be a symplectic surface, and $F$ be a smooth function on $M$ that in local (Darboux) coordinates near some point in $M$ has the form $F (r, \theta) = e^{- f (r, \theta)}$, where
	\[ f (r, \theta) = \frac{4}{r^2 (1 + 15 \cos^2 \theta)} \]
is the composition of the map $(r, \theta) \mapsto \frac{1}{r^2}$ with the area preserving change of coordinates $(x,y) \mapsto (2 x, \frac{y}{2})$.
Here $r \in \R_{\ge 0}$ and $\theta \in \R / 2 \pi \Z$ denote polar coordinates, and $x = r \cos \theta$ and $y = r \sin \theta$ denote rectangular coordinates in the plane.
By cutting off the Hamiltonian $F$ outside a neighborhood of the origin, we may assume it is compactly supported in the domain of the Darboux chart.

For a disk $D^2 \subset \R^2$ contained in the domain of the Darboux chart and centered at the origin, let $\phi_\rho \colon D^2 \to D^2$ be an area preserving homeomorphism compactly supported in the interior of $D^2$, defined by $(r,\theta) \mapsto (r, \theta + \rho (r))$ for $r > 0$, and $\phi_\rho (0) = 0$ at the origin, where $\rho \colon (0,1] \to \R$ is a smooth function with $\rho (r) = 0$ near $r = 1$, cf.\ \cite[Example 4.2]{mueller:ghh07} or \cite[Example 2.6.5]{mueller:ghl08}.
This extends to an area preserving homeomorphism of $M$ by the identity outside $D^2 \subset M$, which is smooth everywhere except at the origin by an appropriate choice of $\rho$.
Indeed, by imposing $\rho > 0$ and $\rho' (r) \to - \infty$ sufficiently fast, $\phi_\rho$ is not even Lipschitz.
Consider the function $H (r,\theta) = F \circ \phi_\rho (r,\theta)$, which is obviously smooth away from the origin.
Since $F$ decays exponentially as $r \to 0^+$, $F \circ \phi_\rho$ converges to zero as $r \to 0^+$.
Similarly, one sees all partial derivatives at the origin exist and vanish, and thus $H$ is a smooth function on $M$.
By Proposition~\ref{pro:uniqueness}, we have $\phi_\rho \circ \phi_H^t \circ \phi_\rho^{-1} = \phi_F^t$, that is, the smooth Hamiltonian vector fields $X_H$ and $X_F$ are topologically conjugate.
If $\rho(r)$ grows like $r^{-a}$ as $r \to 0^+$, where $0 < a < 2$, then $\phi_\rho$ becomes a Hamiltonian homeomorphism that is not Lipschitz.
\end{exa}

Recall that on a smooth manifold $M$ of dimension at most three, every homeomorphism can be approximated uniformly by diffeomorphisms, and if a volume preserving homeomorphism can be approximated uniformly by diffeomorphisms, it can also be approximated uniformly by volume preserving diffeomorphisms.
This in particular means every area preserving homeomorphism is a symplectic homeomorphism.

\begin{lemma}
Suppose that $\rho (r)$ grows like $r^{-a - 1}$ as $r \to 0^+$, where $a > 0$.
If $\psi$ is an area preserving homeomorphism of $M$ that conjugates the flows of $H$ and $F$, then $\psi$ is not Lipschitz; consequently no $C^1$-symplectic diffeomorphism conjugates the Hamiltonian vector fields $X_H$ and $X_F$.
\end{lemma}

\begin{proof}
Because $\psi$ is a symplectic homeomorphism, $F\circ \phi_\rho = H = F \circ \psi$ by Proposition~\ref{pro:uniqueness}, and thus $F\circ (\psi \circ \phi_\rho^{-1}) = F$.
That is, the homeomorphism $\psi \circ \phi_\rho^{-1}$ preserves the level sets of $F$, which near the origin are concentric ellipses centered at the origin.
By hypothesis, there exist two sequences of positive numbers $r_n > r_n' \to 0^+$ such that $ r_n - r_n' < r_n^{2}$, $\rho (r_n) = \pi / 2$ mod $2\pi$, and $\rho (r_n') = \pi$ mod $2 \pi$.
Assuming $\psi$ is Lipschitz, its Lipschitz constant $L > 0$ must obey the inequalities
	\[ L \ge \frac{\left| \psi \left( \phi_\rho^{-1} (r_n, \frac{\pi}{2}) \right) - \psi \left( \phi_\rho^{-1} (r_n',\pi) \right) \right|}{\left| \phi_\rho^{-1} (r_n, \frac{\pi}{2}) - \phi_\rho^{-1} (r_n',\pi) \right|} \ge \frac{r_n - \frac{r_n'}{4}}{r_n - r_n'} > \frac{1}{4} \left( \frac{3}{r_n} + 1 \right) \rightarrow + \infty. \]
The middle inequality holds because $\phi_\rho^{-1} = \phi_{- \rho}$, and $\psi \circ \phi_\rho^{-1}$ preserves the concentric ellipses.
Thus $\psi$ cannot be Lipschitz continuous, and there is no symplectic $C^1$-diffeomorphism conjugating the two Hamiltonian vector fields.
\end{proof}

In fact, the lemma still holds if $\rho(r)$ grows like $r^{-a}$ for $a>0$.
Then there exists $0 < \epsilon < a$, and sequences $r_n$ and $r'_n$ as above, except that $r_n - r'_n < r_n^{a - \epsilon + 1}$.
In this case one obtains the inequality $L > (1/4)(3r_n^{- a + \epsilon} + 1) \to + \infty$.
It is possible to embed countably many disjoint disks (of shrinking radii) into any surface $(M,\omega)$, producing examples where $\psi$ fails to be $C^1$ at at least countably many points.

\begin{exa}\label{exa:non-lipschitz}
Let $(M^{2n},\omega)$ be a symplectic manifold of dimension $2n$.
In local Darboux coordinates $(r_1, \ldots, r_n, \theta_1, \ldots, \theta_n)$, consider the autonomous Hamiltonian
	\[ G (r_1, \ldots, r_n, \theta_1, \ldots, \theta_n) = \int_r^1 s \rho(s)\, ds,\]
where $r = \sqrt{r_1^2 + \cdots r_n^2}$.
For an appropriate choice of $\rho$ as above, this Hamiltonian generates the Hamiltonian homeomorphism $\phi_\rho: M \to M$, given by $\phi_\rho(0) = 0$, and
	\[ \phi_\rho(r_1, \ldots, r_n, \theta_1, \ldots, \theta_n) = (r_1, \ldots, r_n, \theta_1 + \rho(r), \ldots, \theta_n + \rho(r)), \]
c.f.\ \cite{mueller:ghh07, mueller:ghc08, buhovsky:ugh11}.
Define $\overline F =  e^{-f(r_1, \theta_1)}$, where the function $f$ is as in Example~\ref{exa:non-conjugate-surface}, and $\overline H = \overline F \circ \phi_\rho$.
Arguing as above, we obtain two smooth Hamiltonian (and in particular exact divergence-free) vector fields $X_{\overline F}$ and $X_{\overline H}$ on $M$ whose Hamiltonian isotopies are conjugated by a symplectic homeomorphism that is not even Lipschitz.
Moreover, there does not exist a $C^1$-symplectic diffeomorphism nor Lipschitz symplectic homeomorphism conjugating the two isotopies.
\end{exa}

If $M$ is noncompact and $\psi$ is a conformally symplectic $C^1$-diffeomorphism, i.e.\ $\psi^* \omega = c \omega$, then $c F = F \circ (\psi \circ \phi_\rho^{-1})$, and the same argument as above applies to show $\psi$ is not Lipschitz.

\begin{exa}
Let $S^1 \to M^{2 n + 1} \to B^{2 n}$ be the prequantization bundle over a closed symplectic manifold $(B^{2 n},\omega)$ with projection $p \colon M \to B$.
Let $F$ and $H = F \circ \phi_\rho$ be smooth functions on $B$ as in Example~\ref{exa:non-lipschitz}.
Then $\tilde F = F \circ p$ and $\tilde H = H \circ p = F \circ (\phi_\rho \circ p)$ are basic functions on $M$, and thus generate strictly contact isotopies of $M$.
Since $\phi_\rho$ is a Hamiltonian homeomorphism, it induces a well-defined strictly contact homeomorphism $\tilde{\phi}_\rho$ of $M$ with $\phi_\rho \circ p = p \circ \tilde{\phi}_\rho$ \cite{banyaga:ugh11}.
Then $\tilde{H} = F \circ (\phi_\rho \circ p) = F \circ (p \circ \tilde{\phi}_\rho) = \tilde{F} \circ \tilde{\phi}_\rho$, and the strictly contact isotopies of $\tilde{F}$ and $\tilde{H}$ are topologically conjugate by Proposition~\ref{pro:uniqueness}.
If $\tilde{\psi}$ is any other contact diffeomorphism on $M$ (or the uniform limit of strictly contact diffeomorphisms) conjugating the two isotopies, then $e^h \tilde{F} = \tilde{F} \circ (\tilde{\psi} \circ \tilde{\phi}_\rho^{-1})$ for a smooth and thus bounded function $h$ on $M$.
By the same argument as above, any such $\tilde \psi$ has regularity less than Lipschitz.
\end{exa}

If we allow the Hamiltonian vector fields to be time-dependent, we can produce examples of vector fields not conjugated by {\em any} $C^1$-diffeomorphism.

\begin{exa}
Let $f_t (r, \theta)$ be a smooth function on $[0,1] \times \R^2$ that near the origin is given by composition of the function $(r, \theta) \mapsto \frac{1}{r^2}$ with a time-dependent area preserving change of coordinates equal to $(x, y) \mapsto (\sigma (t) x, \frac{y}{\sigma (t)})$, where $\sigma (t) = 2$ near $t = 0$ and $\sigma (t) = \frac{1}{2}$ near $t = 1$, and define a smooth function by $F_t (r, \theta) = e^{- f_t (r, \theta)}$.
Let $\phi_\rho$ be as above, and $H_t = F_t \circ \phi_\rho^{-1}$.
By the same argument as before, this gives rise to smooth Hamiltonian functions on the surface $(M,\omega)$, and their Hamiltonian isotopies are topologically conjugate, $\phi_\rho \circ \phi_F^t = \phi_H^t \circ \phi_\rho$.
\end{exa}

\begin{lemma}
The Hamiltonian isotopies $\{ \phi_F^t\}$ and $\{ \phi_H^t\}$ of $M$ are not $C^1$-conjugate, i.e.\ there exists no $C^1$-diffeomorphism $\psi$ of $M$ such that $\psi \circ \phi_F^t = \phi_H^t \circ \psi$.
\end{lemma}

\begin{proof}
Arguing by contradiction, suppose there exists a $C^1$-diffeomorphism $\psi$ such that $\psi \circ \phi_F^t = \phi_H^t \circ \psi$.
Then in local coordinates, the diffeomorphism $\phi_F^t$ preserves the area form $\psi^* \sigma = \det d\psi(x)\, \sigma$.
Thus $\det d\psi (x) = \det d\psi (\phi_F^t (x))$ for all $x$ near zero.
Near $t=0$ and $t=1$, the flow of $F$ follows concentric ellipses centered at the origin but with major axis at $t=0$ perpendicular to the major axis at $t=1$.
This implies $\det d\psi(x) = c$ is independent of $x$, or $\psi$ is conformally symplectic, at least near the origin.
Since the transformation law is a local statement, we have $c  F_t = H_t \circ \psi$, or $ c F_t \circ (\psi^{-1} \circ \phi_\rho) = F_t$ near the origin.
By essentially the same argument as above, the local inverse $\psi^{-1}$ is not Lipschitz near the origin, a contradiction.
\end{proof}

Note that the examples can be modified so that $F_t$ is $C^\infty$-close to an autonomous Hamiltonian.
The statement that there exists no $C^1$-map $\psi$ such that $\psi \circ \phi_F^t = \phi_H^t \circ \psi$ is false: if $\psi_2$ is the constant map $\psi_2 (p, x) = x_0$, where $x_0$ is any point in $\Sigma$ at which $X_H^t$ vanishes for all $t$ (e.g.\ the origin or a point on the corresponding Reeb circle in the examples above), the above identity holds.

As mentioned in the introduction, one can also define the helicity as follows: for two points $x$ and $y \in M$ and two times $t_1$ and $t_2$, consider the pieces of trajectories $\phi_X^t (x)$, $0 \le t \le t_1$, and $\phi_X^t (y)$, $0 \le t \le t_2$, and close them up to loops using a `system of short paths' in $M$.
The asymptotic linking number of these two loops is defined, and the helicity equals the average of these asymptotic linking numbers over $M \times M$.
See for example \cite{ghys:kd07} for details.
This alternate definition suggests that the helicity could be invariant under topological conjugation.
However, the system of short paths considered above to close up the pieces of trajectories may become tangled up when conjugating with a homeomorphism, so invariance of the helicity is not obvious.
We now observe that the problem is indeed a local one.

\begin{lemma}
Let $\{ U_i \}$ be an open cover of a closed smooth three-manifold $M$ with volume form $\mu$.
The helicity is invariant under conjugation by volume preserving homeomorphisms, if and only if it is invariant under conjugation by volume preserving homeomorphisms that are isotopic to the identity through isotopies of volume preserving homeomorphisms that are supported in an open set $U_i$ (and with vanishing mass flow).
\end{lemma}

\begin{proof}
As we noted before, in dimension less or equal to three, every homeomorphism can be approximated uniformly by diffeomorphisms, and a homeomorphism $\phi$ that preserves the measure induced by $\mu$ can be approximated uniformly by volume preserving diffeomorphisms.
The group $\Homeo (M,\mu)$ of volume preserving homeomorphisms of $M$ is locally path-connected in the compact-open topology \cite{fathi:sgh80}.
Thus there exists a volume preserving diffeomorphism $\psi$ sufficiently close to $\phi$ that they are isotopic inside $\Homeo (M,\mu)$, or equivalently, the volume preserving homeomorphism $\phi \circ \psi^{-1}$ is isotopic to the identity in $\Homeo (M,\mu)$.
The helicity is invariant under conjugation by $\phi = (\phi \circ \psi^{-1}) \circ \psi$ if and only if it is invariant under conjugation by $\phi \circ \psi^{-1}$ (Lemma~\ref{lem:change-vol-form}).
Thus without loss of generality we may assume $\phi$ is isotopic to the identity through an isotopy of volume preserving homeomorphisms.
If that is the case, its mass flow is well-defined \cite{fathi:sgh80}.
There is a dual homomorphism, the flux mentioned in Section~\ref{sec:helicity}, for isotopies of volume preserving diffeomorphisms \cite{banyaga:scd97}.
By surjectivity of the flux, there exists an isotopy of volume preserving diffeomorphisms with the same mass flow as the isotopy connecting $\phi$ to the identity.
By the same argument as above, the general case reduces to considering volume preserving homeomorphisms $\phi$ with vanishing mass flow.
Such homeomorphisms can be fragmented into a finite composition $\phi = \phi_m \circ \cdots \circ \phi_1$ of volume preserving homeomorphisms so that each $\phi_k$ is supported (and isotopic to the identity with vanishing mass flow) inside an element of an open cover of $M$ \cite{fathi:sgh80}.
\end{proof}

For example, one may choose as subsets the domains of a Darboux atlas with respect to a contact form on $M$.
We note that a volume preserving diffeomorphism with vanishing flux or mass flow may also be fragmented into diffeomorphisms with `small' support, however, the helicity is not a homomorphism.

\section{Higher-dimensional helicities} \label{sec:higher-dim}
There are several generalizations of helicity to higher dimensions studied for example in \cite{khesin:ghh03, kotschick:lnm03, riviere:hdh02}, see also \cite[Chapter III, 7.B]{arnold:tmh98}.
In \cite{kotschick:lnm03}, the authors consider the linking number of a divergence-free vector field on a manifold of arbitrary dimension with a codimension two foliation endowed with an invariant transverse measure.
In this short section we compute this linking number for a strictly contact vector field on a regular contact manifold.
This simultaneously generalizes Examples 3.8 (Hamiltonian vector fields on closed symplectic manifolds) and 3.9 (Reeb vector fields on closed contact manifolds) in \cite{kotschick:lnm03}.

\begin{prop}
Let $(M^{2 n + 1},\alpha)$ be a closed manifold together with a regular contact form $\alpha$, and let $(B^{2n},\omega)$ be the base of the corresponding Boothby-Wang bundle.
Suppose $A \subset B$ is a closed, oriented, and null-homologous codimension two submanifold, and denote $N = p^{-1} (A) \subset M$, where $p \colon M \to B$ is the projection.
Suppose further $X_H$ is a strictly contact vector field on $(M,\alpha)$, and write $F$ for the unique smooth function on $B$ satisfying $p^* F = H$.
Then
	\[ \Hel (X_H,N) = - n \int_A F \, \omega^{n - 1}. \]
\end{prop}

This number obviously extends to an invariant of continuous strictly contact isotopies on $(M,\alpha)$, and is invariant under conjugation by uniform limits of strictly contact diffeomorphism, provided that limit preserves $N$.

\begin{proof}
By definition \cite{kotschick:lnm03},
	\[ \Hel (X_H,N) = \int_N \beta, \]
where $\beta$ is a primitive of $\iota_X \mu$, and $\mu$ is the canonical volume form on $M$ induced by $\alpha$.
Since $\int_N \tau = \int_A p_* \tau = 0$ for any closed $(2 n - 1)$-form $\tau$ on $M$ ($A$ is null-homologous), $N$ is null-homologous as well, so that the above integral is well-defined.
Here $p_* \colon H^{2 n - 1} (M) \to H^{2 n - 2} (B)$ is the induced map in the Gysin sequence of the $S^1$-bundle $S^1 \to M \to B$.
Recall by (\ref{eqn:beta}),
	\[ \beta = (n + 1)\, p^* \gamma + ((n + 1) c_H  - n H)\, \alpha \wedge (d\alpha)^{n - 1}, \]
so that by the same partition of unity argument as above, and since $p_{*}R_\alpha = 0$,
	\[ \int_N \beta = \int_A ((n + 1)\, c_H  - n F)\, \omega^{n - 1} = - n \int_A F\, \omega^{n - 1} \]
as claimed.
\end{proof}

See \cite{kotschick:lnm03} for further replacing the submanifold $A$ by an oriented (possibly singular) codimension two foliation $\mathcal F$ with a holonomy-invariant transverse measure.

\appendix

\section{Proof of Proposition~\ref{pro:homotopy-groups}} \label{sec:proof}
\begin{proof}
The short exact sequence (\ref{eqn:s^1-ext}) is a Serre fibration.
Indeed, a homotopy $\Phi_t \colon D^k \to \Ham (M,\omega)$, $0 \le t \le 1$, with respect to the $C^\infty$-topology on $\Ham (B,\omega)$, is a smooth map $D^k \times [0,1] \to \Ham (B,\omega)$, $(s,t) \mapsto \phi_{s,t} \in \Ham (B,\omega)$.
There exists a unique family $F_{s,t} = F_s (t, \cdot)$ of (normalized) smooth Hamiltonian functions such that $\phi_{s,t} = \phi_{F_s}^t \circ \phi_{s,0}$, where $t \mapsto \phi_{F_s}^t$ is the Hamiltonian flow of $F_s$ starting at the identity.
Define $H_{s,t} = p^* F_{s,t} = F_{s,t} \circ p \colon M \to \R$ for $0 \le t \le 1$ and $s \in D^k$, and denote by $t \mapsto \psi_{H_s}^t$ the strictly contact flow of $H_s$ starting at the identity.
Given a lift $\Psi_0 \colon D^k \to \Diff_0 (M,\alpha)$, say $s \mapsto \psi_{s,0}$, of $\Phi_0$, the homotopy $\Psi_t \colon D^k \to \Diff_0 (M,\alpha)$ defined by $s \mapsto \psi_{s,t} = \psi_{H_s}^t \circ \psi_{s,0} \in \Diff (M,\alpha)$, lifts the homotopy $\Phi_t$.
Thus (\ref{eqn:s^1-ext}) has the homotopy lifting property with respect to all disks.

Since $\Ham (B,\omega)$ is path-connected, (\ref{eqn:s^1-ext}) gives rise to a long exact sequence of homotopy groups
\begin{equation} \label{eqn:les}
	\cdots \rightarrow \pi_k (S^1) \rightarrow \pi_k ( \Diff_0 (M,\alpha) ) \rightarrow \pi_k ( \Ham (B,\omega) ) \rightarrow \pi_{k - 1} (S^1) \rightarrow \cdots
\end{equation}

According to \cite[Section 7.2]{polterovich:ggs01}, if $B$ is a closed and connected surface, then $\Ham (B,\omega)$ is $\Z_2$ if $B = S^2$, with generator the one-turn rotation of the sphere, and trivial otherwise.
Moreover, $\pi_2 ( \Ham (B,\omega) ) = 0$ for all closed surfaces $B$.
To see this, recall the inclusion $\Ham (B,\omega) \hookrightarrow \Symp_0 (B,\omega)$ induces an isomorphism on homotopy groups $\pi_k$ for $k > 1$ \cite[Section 10.2]{mcduff:ist98}.
Moser's argument shows that for closed surfaces $B$ the inclusion of $\Symp_0 (B,\omega)$ into $\Diff_0 (B)$ induces isomorphisms on all homotopy groups, see \cite[Section 1.5]{banyaga:scd97} or again \cite{polterovich:ggs01}.
Moreover, the latter is contractible for genus at least $2$, and has strong deformation retract $T^2$ and $SO (3)$ for genus $1$ and $0$, respectively \cite{earle:dgc67}.
Now $\pi_k (T^2)$ obviously vanishes for $k > 1$, and since $SO (3) \cong \R P^3$ has universal covering space $S^3$, we get $\pi_k (SO (3)) \cong \pi_k (S^3)$ for $k > 1$.
Combining these facts, we see that $\pi_2 ( \Ham (B,\omega) )$ indeed vanishes for all closed surfaces $B$.

From the long exact sequence (\ref{eqn:les}) we obtain the description of the fundamental group of $\Diff (M,\alpha)$.
For $k = 2$, recall $\pi_2 ( \Ham (B,\omega) )$ and $\pi_2 (S^1)$ are zero, and therefore $\pi_2 ( \Diff (M,\alpha) ) = 0$.
For $k > 2$, the homotopy groups $\pi_k (S^1)$ and $\pi_{k - 1} (S^1)$ to the left and right vanish, thus $\pi_k ( \Diff (M,\alpha) ) \cong \pi_k ( \Ham (B,\omega) )$, and the claim follows from the same argument as above.
\end{proof} 

\section{A non-regular contact three-manifold} \label{sec:torus}
By Martinet's theorem, any closed three-manifold admits a contact structure.
On the other hand, the torus $T^3$ does not admit a regular contact form.
Consider the contact form $\alpha = \cos z \, dx - \sin z \, dy$ on $T^3$ with induced volume form $dx \wedge dy \wedge dz$, where $x, y, z \in \R / (2 \pi \Z)$.
A basic function on $(T^3,\alpha)$ depends only on $z$ \cite{mueller:fsc11} and thus can be written as a Fourier series
	\[ H (z) = \sum_{n \in \Z} c_n e^{i n z} = \sum_{n = 0}^\infty a_n \cos (n z) + b_n \sin (n z). \]
Its strictly contact vector field $X_H$ has flux
	\[ [ \iota_{X_H} (\alpha \wedge d\alpha) ] = a_1 [ dy \wedge dz ] + b_1 [ dx \wedge dz ], \]
and thus $X_H$ is exact if and only if $a_1 = 0 = b_1$, or equivalently, $c_1 = 0$ \cite{mueller:fsc11}.
\begin{prop}
The helicity of a strictly contact vector field $X_H$ whose contact Hamiltonian $H:T^3 \to \R$ satisfies $c_1 = 0$ is given by
\begin{align}\label{eqn:helicity-three-torus} 
\Hel (X_H) = - \sum_{n \in \Z} ( 3 + \frac{4}{n^2 - 1} ) | c_n |^2 = c_0^2 - 2 \sum_{n > 0} ( 3 + \frac{4}{n^2 - 1} ) | c_n |^2.
\end{align}
In particular, the helicity is bounded from above and below by a multiple of $\| H \|_{L^2}^2$.
\end{prop}
\begin{proof}
If $c_1 = 0$, then also $c_{-1} = \overline{c}_1 = 0$, and we can define \emph{real} functions
	\[ F (z) = \sum_{n \in \Z} c_n \left( \frac{1}{n + 1} e^{i (n + 1) z} - \frac{1}{n - 1} e^{i (n - 1) z} \right), \]
	\[ G (z) = \sum_{n \in \Z} i c_n \left( \frac{1}{n + 1} e^{i (n + 1) z} + \frac{1}{n - 1} e^{i (n - 1) z} \right). \]
Then
	\[ d (F dx + G dy - H \alpha) = 2 H d\alpha - d (H \alpha) = \iota_{X_H} (\alpha \wedge d\alpha), \]
and a direct computation shows
\begin{eqnarray*}
	&& (F dx + G dy - H \alpha) \wedge d (F dx + G dy - H \alpha) \\
	&& \hspace{1cm} = (2 H (F \cos z - G \sin z) - 3 H^2) \, dx \wedge dy \wedge dz + \text{exact terms},
\end{eqnarray*}
and
	\[ 2 H (F \cos z - G \sin z) - 3 H^2 = \sum_{n,m \in \Z} \left( - \frac{4}{n^2 - 1} - 3 \right) c_n c_m e^{i (n + m) z}. \]
We note that the constant term is $c (H) = c_0 = \int_{T^3} H \, dx \wedge dy \wedge dz$, and also $\sum_{n \in \Z} | c_n |^2 = \| H \|_{L^2}^2$, and thus (\ref{eqn:helicity-three-torus}) follows.
\end{proof}

\bibliography{helicity-arxiv}
\bibliographystyle{amsalpha}
\end{document}